\documentclass[12pt]{amsart}
\usepackage[letterpaper, margin=1in]{geometry}
\usepackage{amsmath, amsthm, amssymb,color, mathtools}
\usepackage[linesnumbered,ruled]{algorithm2e}
\usepackage{graphicx}
\usepackage{amsfonts,amssymb,amscd,latexsym,txfonts,stmaryrd,enumerate}
\usepackage[colorlinks=true, linkcolor=blue, anchorcolor=black, citecolor=blue, filecolor=blue, menucolor= blue, urlcolor=black]{hyperref}

\usepackage{epstopdf}
\usepackage[all]{xypic}
\usepackage{verbatim}
\usepackage[dvipsnames]{xcolor}
\usepackage[final]{pdfpages}
\usepackage{bm}

\usepackage{makecell}
\usepackage{array}
\newcolumntype{?}{!{\vrule width 1pt}}
\newif\ifincludeprevious

\includeprevioustrue  

\newtheorem{theorem}{Theorem}
 \numberwithin{theorem}{section}
\newtheorem{proposition}[theorem]{Proposition}
\newtheorem{lemma}[theorem]{Lemma}
\newtheorem{corollary}[theorem]{Corollary}

\theoremstyle{definition}
\newtheorem{definition}[theorem]{Definition}

\newtheorem{remark}[theorem]{Remark}
\newtheorem{example}[theorem]{Example}

\usepackage{tikz}
\usetikzlibrary{calc}
\usetikzlibrary{arrows}
\usetikzlibrary{decorations.pathmorphing}

\newcommand{\xx}{1}
\newcommand{\yy}{1}

\newcommand{\bt}{\tikz{\node[shape=circle,draw,inner sep=2pt] {};}}
\newcommand{\ls}[1]{{\footnotesize{#1}}}

\newcommand{\stage}[2]{\tikz[baseline=(char.base)]{
            \node[shape=circle,draw,inner sep=0.5pt,fill={#1}] (char) {#2};}}

\newcommand{\ifi}{\ensuremath{\ \Leftrightarrow}\ }
\newcommand{\ds}{\displaystyle}
\newcommand{\R}{\mathbb{R}}

\newcommand{\Z}{\mathbb{Z}}
\newcommand{\T}{\mathcal{T}}

\newcommand{\B}{\mathcal{B}}
\newcommand{\A}{\mathcal{A}}
\newcommand{\M}{\mathcal{M}}
\newcommand{\la}{\mathcal{L}}
\newcommand{\aaa}{\boldsymbol{a}}
\newcommand{\bbb}{\boldsymbol{b}}
\newcommand{\pp}{\mathbb{P}}

\newcommand{\C}{\mathbb{C}}

\DeclareMathOperator{\im}{im}

\newcommand{\ipaths}{\ensuremath{I_{\mathrm{Paths}}}}

\newcommand{\degree}{\mathrm{deg}}
\newcommand{\Quad}{\mathrm{Quad}}
\newcommand{\Lift}{\mathrm{Lift}}

\renewcommand{\>}{\rangle}

\newcommand\indep{\protect\mathpalette{\protect\independenT}{\perp}}
\newcommand\independent{\protect\mathpalette{\protect\independenT}{\perp}}
\def\independenT#1#2{\mathrel{\rlap{$#1#2$}\mkern2mu{#1#2}}}

\date{}
 
\title{\textbf{Gr\"obner bases for staged trees}
}

\author{Lamprini Ananiadi and Eliana Duarte}
\date{\small \today}                                           

\begin{document}


\begin{abstract}
In this article we consider the problem of finding generators of
 the toric ideal associated to a combinatorial object called a staged tree. Our main theorem
 states that toric ideals
of staged trees that are balanced and stratified are generated
by a quadratic Gr\"obner basis whose initial ideal is squarefree.  The proof of this result is based on Sullivant's \cite{Sullivant2007}
toric fiber product construction. 
\end{abstract}


\maketitle

\section{Introduction}

 The study of toric ideals associated to statistical models was pioneered by the work of Diaconis and Sturmfels
 \cite{DiaconisSturmfels} who first used the generators of a toric ideal to formulate 
 a sampling algorithm for discrete distributions. Since then, and with the subsequent work of \cite{Sullivant2006,Sullivant2007} and \cite{GMS}
 the study of toric ideals of discrete statistical models has been an active area of research in 
 Algebraic Statistics.  The books by Sullivant \cite[Chapter 9]{Sullivant2019} and Aoki, Hara and Takemura \cite{Aoki2012} are good references 
 to learn about the role of toric ideals in statistics. A recent introduction to the topic from the point of view of binomial ideals can be found in \cite[Chapter 9]{BinomialIdeals}, which also contains a thorough list of references of previous contributions to this topic.

In 2008, Smith and Anderson \cite{SmithAnderson} introduced a new graphical discrete
statistical model called a \emph{staged tree model}. This model is represented
by an event tree together with an equivalence relation on its vertices. 
Staged tree models are useful to
represent conditional independence relations among events. In particular we can use staged tree models to 
represent some
conditional independence statements between random variables. For example those coming from graphical
models such as Bayesian networks and decomposable models.  Hence  any discrete 
Bayesian network or decomposable model is also a staged tree model \cite{SmithAnderson}. There are two properties that make
staged tree models
more general than Bayesian networks or decomposable models. The first is that the state space of a staged tree
model does not have to be a cartesian product. The second is that using staged
tree models it is possible to represent extra context-specific conditional independence between events.
The book of Collazo, G\"orgen and Smith \cite{CEGbook} is a good reference to learn about these models.

In this article we define the toric 
ideal associated to a \emph{staged tree} and study its properties from an algebraic and combinatorial point of view.
We present Theorem~\ref{thm:main}
which states that toric ideals of staged trees  that are \emph{balanced} and \emph{stratified} have  quadratic Gr\"obner 
basis with squarefree initial ideal. 
We apply Theorem~\ref{thm:main} in Section~\ref{sec:stats} to obtain Gr\"obner 
bases for toric ideals of staged tree models. Our results provide a
 new point of view on the construction of Gr\"obner bases for
decomposable graphical models, some conditional independence models as well as the construction
of Gr\"obner bases for staged tree models whose underlying tree is \emph{asymmetric}.

%
%
%

This article is organized as follows. In Section~\ref{sec:stagedtrees} we define the toric ideal 
associated
to a staged tree. In Section~\ref{sec:tfps} we formulate a toric fiber product construction for
balanced and stratified staged trees. In Section~\ref{sec:proofs} we prove our main result Theorem~\ref{thm:main}. 
 Finally in Section~\ref{sec:stats} we apply our results to compute Gr\"obner bases for
 several statistical models.

\section{Staged trees } \label{sec:stagedtrees}
We start by defining our two objects of interest:  a staged tree and its associated toric
ideal.
Let $\T = (V,E)$ denote a directed rooted
tree graph with vertex set $V$ and edge set $E$.
For $v,w\in V$ the directed edge in $E$ from $v$ to $w$ is denoted by $(v,w)$, 
the set of children of $v$ is $\mathrm{ch}(v)=\{u \mid (v,u)\in E \}$, and the set
of outgoing edges from $v$ is $E(v)=\{(v,u)\,|\,u\in \mathrm{ch}(v) \}$.
Given a set $\la$ of labels, to each $e\in E$ we associate a label from $\la$
via the rule $\theta: E \to \la$. We require that $\theta$ is surjective.
For each vertex $v\in V$,  we let  $\theta_{v}:=\{\theta(e)\,\mid\,
e\in E(v)\}$ be the set of labels associated to $v$.

\begin{definition}\label{def:stagedtree} Let $\la$ be a set of labels.
A tree $\T=(V,E)$ together with a labelling $\theta:E\to \la$ is a \emph{staged tree} if: (1) for each $v\in V,$
$|\theta_{v}|=|E(v)|$, and (2) for any two vertices $v,w \in V$ the sets $\theta_{v}, \theta_{w}$ are either equal
or disjoint.
\end{definition}

 Conditions (1) and (2) in Definition~\ref{def:stagedtree} of a staged tree define an equivalence relation on the set of vertices of $\T$. Namely
 $v,w \in V$ are equivalent if and only if $\theta_{v}= \theta_{w}$. 
 We refer to the partition induced by
 this equivalence relation on the set $V$ as the set of stages 
of $\T$ and to a single element in this partition as a \emph{stage}.  We use $(\T,\theta)$ to denote
a staged tree with labeling rule $\theta$. For simplicity we will often drop the use of $\theta$ and
write $\T$ for a staged tree.

To define the toric ideal associated to $(\T,\theta)$ we define two polynomial rings. The first ring is $\R[p]_{\T}:=\R[p_{\lambda} \,|\, \lambda
\in \Lambda]$ where $\Lambda$ is the set of root-to-leaf 
paths in $\T$. The second ring is
$\R[\Theta]_{\T}:=\R[z,\la]$ where the labels in $\la$ are 
indeterminates
together with a homogenizing variable $z$. For a directed or undirected path $\gamma$ in $\T$, 
$E(\gamma)$ is the set of edges in $\gamma$.

\begin{definition} \label{def:toricideal}
The \emph{toric staged tree ideal} associated to $(\T,\theta)$ is the kernel of the ring homomorphism  $\varphi_{\T}:\R[p]_{\T} \rightarrow \R[\Theta]_
{\T}$ defined as
\begin{align}\label{eq:param}
  p_{\lambda} &\mapsto \ds z\cdot \prod_{e\in E(\lambda)} \theta(e)\,\,.
\end{align}
\end{definition}
If $n=|\la|$ is the number of distinct edge labels in $\T$, then  $\ker(\varphi_{\T})$ is the defining 
ideal of the projective toric variety defined by the closure of the image of the
monomial parameterization $\Phi_{\T}:(\C^*)^n\to \pp^{|\Lambda|-1}$ given by
$(\theta(e)\,|\, \theta(e)\in \im(\theta)) \mapsto \prod_{e\in E(\lambda)} \theta(e)$.

\begin{example} \label{ex:deco}
The staged tree $\T_1$ in Figure~\ref{fig:trees} has label set
 $\la=\{s_0,\ldots,s_{13}\}$. Each vertex in $\T_1$ is denoted by a string of $0$'s and $1$'s and each edge has a label associated to it.
 The root-to-leaf paths in $\T$ are  denoted by $p_{ijkl}$ where $i,j,k,l\in \{0,1\}$. A vertex in $\T_1$
represented with a blank circle indicates a stage consisting of a single vertex. We use colors in the
 vertices
of $\T_1$ to indicate which vertices are in the same stage. For instance all the purple vertices, i.e. the set of vertices 
$\{000,010,100,110\}$, are
in the same stage and therefore they 
have the same set $\{s_{10},s_{11}\}$ of associated edge labels. The map $\Phi_{\T}$ sends
$(s_0,\ldots,s_{13})$ to
\begin{align*}
 ({s}_{0}{s}_{2}{s}_{6}{s}_{10},\,{s}_{0}{s}_{2}{s}_{6}{s}_{11},\,{s}_{0}{s}_{2}{s}_{7}{s}_{12},
     \,{s
      }_{0}{s}_{2}{s}_{7}{s}_{13},\,{s}_{0}{s}_{3}{s}_{8}{s}_{10},\,{s}_{0}{s}_{3}{s}_{8}{s}_{11},\,{s}_{0}{s}_{3
      }{s}_{9}{s}_{12},\, 
       {s}_{0}{s}_{3}{s}_{9}{s}_{13},\,\\{s}_{1}{s}_{4}{s}_{6}{s}_{10},\,{s}_{1}{s}_{4}{s}_{6}{s
      }_{11},\,{s}_{1}{s}_{4}{s}_{7}{s}_{12},\,{s}_{1}{s}_{4}{s}_{7}{s}_{13},\,{s}_{1}{s}_{5}{s}_{8}{s}_{10},\,{s
      }_{1}{s}_{5}{s}_{8}{s}_{11},\, 
       {s}_{1}{s}_{5}{s}_{9}{s}_{12},\,{s}_{1}{s}_{5}{s}_{9}{s}_{13}).
\end{align*}
 The toric ideal $\ker(\varphi_\T)$ is generated by a quadratic Gr\"obner basis with squarefree initial ideal.

\end{example}

\begin{example}
Consider the two staged trees $\T_2,\T_3$ depicted in Figure~\ref{fig:trees}. For the staged tree $\T_2$,
$\ker(\varphi_{\T_2})$ is
generated
by a quadratic Gr\"obner basis with squarefree intial ideal. For $\T_3$, the ideal $\ker(\varphi_{\T_3})$ 
also has a Gr\"obner
basis with squarefree initial ideal but its elements are of degree two and four.
\end{example}
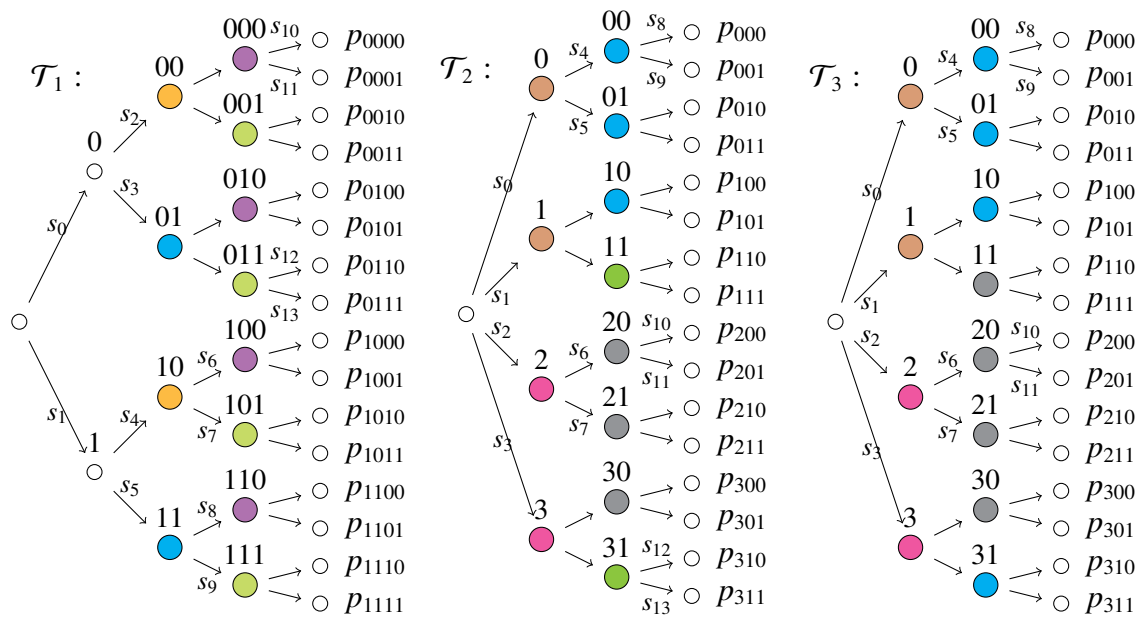
\begin{figure}[t]
\begin{center}
\begin{tikzpicture}
\renewcommand{\xx}{1}
\renewcommand{\yy}{0.5}
\node at (1.5*\xx,14*\yy) {$\T_1:$};

\node (r) at (1*\xx,7.5*\yy) {\bt};

\node (dd1) at (2*\xx,12.3*\yy) {\small $0$};
\node (dd2) at (2*\xx,4.3*\yy) {\small$1$};

\node (d1) at (2*\xx,11.5*\yy) {\bt};
\node (d2) at (2*\xx,3.5*\yy) {\bt};

\node (c1) at (3*\xx,14.3*\yy) {\small$00$};
\node (c2) at (3*\xx,10.3*\yy) {\small$01$};
\node (c3) at (3*\xx,6.3*\yy) {\small$10$};
\node (c4) at (3*\xx,2.3*\yy) {\small$11$};

\node (c1) at (3*\xx,13.5*\yy) {\stage{Dandelion}{$\phantom{;}$}};
\node (c2) at (3*\xx,9.5*\yy) {\stage{ProcessBlue}{$\phantom{;}$}};
\node (c3) at (3*\xx,5.5*\yy) {\stage{Dandelion}{$\phantom{;}$}};
\node (c4) at (3*\xx,1.5*\yy) {\stage{ProcessBlue}{$\phantom{;}$}};

\node (b1) at (4*\xx,15.3*\yy) {\small $000$};
\node (b2) at (4*\xx,13.3*\yy) {\small $001$};
\node (b3) at (4*\xx,11.3*\yy) {\small $010$};
\node (b4) at (4*\xx,9.3*\yy) {\small $011$};
\node (b5) at (4*\xx,7.3*\yy) {\small $100$};
\node (b6) at (4*\xx,5.3*\yy) {\small $101$};
\node (b7) at (4*\xx,3.3*\yy) {\small $110$};
\node (b8) at (4*\xx,1.3*\yy) {\small $111$};

\node (b1) at (4*\xx,14.5*\yy) {\stage{Orchid}{$\phantom{;}$}};
\node (b2) at (4*\xx,12.5*\yy) {\stage{SpringGreen}{$\phantom{;}$}};
\node (b3) at (4*\xx,10.5*\yy) {\stage{Orchid}{$\phantom{;}$}};
\node (b4) at (4*\xx,8.5*\yy) {\stage{SpringGreen}{$\phantom{;}$}};
\node (b5) at (4*\xx,6.5*\yy) {\stage{Orchid}{$\phantom{;}$}};
\node (b6) at (4*\xx,4.5*\yy) {\stage{SpringGreen}{$\phantom{;}$}};
\node (b7) at (4*\xx,2.5*\yy) {\stage{Orchid}{$\phantom{;}$}};
\node (b8) at (4*\xx,0.5*\yy) {\stage{SpringGreen}{$\phantom{;}$}};

\node (a1) at (5*\xx,15*\yy) {\bt};
\node (a2) at (5*\xx,14*\yy) {\bt};
\node (a3) at (5*\xx,13*\yy) {\bt};
\node (a4) at (5*\xx,12*\yy) {\bt};
\node (a5) at (5*\xx,11*\yy) {\bt};
\node (a6) at (5*\xx,10*\yy) {\bt};
\node (a7) at (5*\xx,9*\yy) {\bt};
\node (a8) at (5*\xx,8*\yy) {\bt};
\node (a9) at (5*\xx,7*\yy) {\bt};
\node (a10) at (5*\xx,6*\yy) {\bt};
\node (a11) at (5*\xx,5*\yy) {\bt};
\node (a12) at (5*\xx,4*\yy) {\bt};
\node (a13) at (5*\xx,3*\yy) {\bt};
\node (a14) at (5*\xx,2*\yy) {\bt};
\node (a15) at (5*\xx,1*\yy) {\bt};
\node (a16) at (5*\xx,0*\yy) {\bt};

\draw[->] (r) -- node [above] {\ls{$s_0$}} (d1);
\draw[->] (r) -- node [below] {\ls{$s_1$}} (d2);

\draw[->] (d1) -- node [above] {\ls{$s_2$}} (c1);
\draw[->] (d1) -- node [above] {\ls{$s_3$}} (c2);
\draw[->] (d2) -- node [above] {\ls{$s_4$}} (c3);
\draw[->] (d2) -- node [above] {\ls{$s_5$}} (c4);

\draw[->] (c1) -- node [above] {} (b1);
\draw[->] (c1) -- node [below] {} (b2);
\draw[->] (c2) -- node [above] {} (b3);
\draw[->] (c2) -- node [below] {} (b4);
\draw[->] (c3) -- node [above] {\ls{$s_6$}} (b5);
\draw[->] (c3) -- node [below] {\ls{$s_7$}} (b6);
\draw[->] (c4) -- node [above] {\ls{$s_8$}} (b7);
\draw[->] (c4) -- node [below] {\ls{$s_9$}} (b8);

\draw[->] (b1) -- node [above] {\ls{$s_{10}$}} (a1);
\draw[->] (b1) -- node [below] {\ls{$s_{11}$}} (a2);
\draw[->] (b2) -- node [above] {} (a3);
\draw[->] (b2) -- node [below] {} (a4);
\draw[->] (b3) -- node [above] {} (a5);
\draw[->] (b3) -- node [above] {} (a6);
\draw[->] (b4) -- node [above] {\ls{$s_{12}$}} (a7);
\draw[->] (b4) -- node [below] {\ls{$s_{13}$}} (a8);
\draw[->] (b5) -- node [above] {} (a9);
\draw[->] (b5) -- node [above] {} (a10);
\draw[->] (b6) -- node [above] {} (a11);
\draw[->] (b6) -- node [above] {} (a12);
\draw[->] (b7) -- node [above] {} (a13);
\draw[->] (b7) -- node [above] {} (a14);
\draw[->] (b8) -- node [above] {} (a15);
\draw[->] (b8) -- node [above] {} (a16);


\node [right, xshift=5] at (a1) {$p_{0000}$};
\node [right, xshift=5] at (a2) {$p_{0001}$};
\node [right, xshift=5] at (a3) {$p_{0010}$};
\node [right, xshift=5] at (a4) {$p_{0011}$};
\node [right, xshift=5] at (a5) {$p_{0100}$};
\node [right, xshift=5] at (a6) {$p_{0101}$};
\node [right, xshift=5] at (a7) {$p_{0110}$};
\node [right, xshift=5] at (a8) {$p_{0111}$};
\node [right, xshift=5] at (a9) {$p_{1000}$};
\node [right, xshift=5] at (a10){$p_{1001}$};
\node [right, xshift=5] at (a11){$p_{1010}$};
\node [right, xshift=5] at (a12){$p_{1011}$};
\node [right, xshift=5] at (a13){$p_{1100}$};
\node [right, xshift=5] at (a14){$p_{1101}$};
\node [right, xshift=5] at (a15){$p_{1110}$};
\node [right, xshift=5] at (a16){$p_{1111}$};

\end{tikzpicture}\,
\begin{tikzpicture}
\renewcommand{\xx}{1}
\renewcommand{\yy}{0.5}
\node at (2*\xx,14*\yy) {$\T_2:$};

\node (r) at (2*\xx,7.5*\yy) {\bt};

\node (c1) at (3*\xx,14.3*\yy) {\small$0$};
\node (c2) at (3*\xx,10.3*\yy) {\small$1$};
\node (c3) at (3*\xx,6.3*\yy) {\small$2$};
\node (c4) at (3*\xx,2.3*\yy) {\small$3$};

\node (c1) at (3*\xx,13.5*\yy) {\stage{Tan}{$\phantom{;}$}};
\node (c2) at (3*\xx,9.5*\yy) {\stage{Tan}{$\phantom{;}$}};
\node (c3) at (3*\xx,5.5*\yy) {\stage{VioletRed}{$\phantom{;}$}};
\node (c4) at (3*\xx,1.5*\yy) {\stage{VioletRed}{$\phantom{;}$}};

\node (b1) at (4*\xx,15.3*\yy) {\small $00$};
\node (b2) at (4*\xx,13.3*\yy) {\small $01$};
\node (b3) at (4*\xx,11.3*\yy) {\small $10$};
\node (b4) at (4*\xx,9.3*\yy) {\small $11$};
\node (b5) at (4*\xx,7.3*\yy) {\small $20$};
\node (b6) at (4*\xx,5.3*\yy) {\small $21$};
\node (b7) at (4*\xx,3.3*\yy) {\small $30$};
\node (b8) at (4*\xx,1.3*\yy) {\small $31$};

\node (b1) at (4*\xx,14.5*\yy) {\stage{Cyan}{$\phantom{;}$}};
\node (b2) at (4*\xx,12.5*\yy) {\stage{Cyan}{$\phantom{;}$}};
\node (b3) at (4*\xx,10.5*\yy) {\stage{Cyan}{$\phantom{;}$}};
\node (b4) at (4*\xx,8.5*\yy) {\stage{LimeGreen}{$\phantom{;}$}};
\node (b5) at (4*\xx,6.5*\yy) {\stage{Gray}{$\phantom{;}$}};
\node (b6) at (4*\xx,4.5*\yy) {\stage{Gray}{$\phantom{;}$}};
\node (b7) at (4*\xx,2.5*\yy) {\stage{Gray}{$\phantom{;}$}};
\node (b8) at (4*\xx,0.5*\yy) {\stage{LimeGreen}{$\phantom{;}$}};

\node (a1) at (5*\xx,15*\yy) {\bt};
\node (a2) at (5*\xx,14*\yy) {\bt};
\node (a3) at (5*\xx,13*\yy) {\bt};
\node (a4) at (5*\xx,12*\yy) {\bt};
\node (a5) at (5*\xx,11*\yy) {\bt};
\node (a6) at (5*\xx,10*\yy) {\bt};
\node (a7) at (5*\xx,9*\yy) {\bt};
\node (a8) at (5*\xx,8*\yy) {\bt};
\node (a9) at (5*\xx,7*\yy) {\bt};
\node (a10) at (5*\xx,6*\yy) {\bt};
\node (a11) at (5*\xx,5*\yy) {\bt};
\node (a12) at (5*\xx,4*\yy) {\bt};
\node (a13) at (5*\xx,3*\yy) {\bt};
\node (a14) at (5*\xx,2*\yy) {\bt};
\node (a15) at (5*\xx,1*\yy) {\bt};
\node (a16) at (5*\xx,0*\yy) {\bt};

\draw[->] (r) -- node [above] {\ls{$s_0$}} (c1);
\draw[->] (r) -- node [below] {\ls{$s_1$}} (c2);
\draw[->] (r) -- node [above] {\ls{$s_2$}} (c3);
\draw[->] (r) -- node [below] {\ls{$s_3$}} (c4);

\draw[->] (c1) -- node [above] {\ls{$s_4$}} (b1);
\draw[->] (c1) -- node [below] {\ls{$s_5$}} (b2);
\draw[->] (c2) -- node [above] {} (b3);
\draw[->] (c2) -- node [below] {} (b4);
\draw[->] (c3) -- node [above] {\ls{$s_6$}} (b5);
\draw[->] (c3) -- node [below] {\ls{$s_7$}} (b6);
\draw[->] (c4) -- node [above] {} (b7);
\draw[->] (c4) -- node [below] {} (b8);

\draw[->] (b1) -- node [above] {\ls{$s_{8}$}} (a1);
\draw[->] (b1) -- node [below] {\ls{$s_{9}$}} (a2);
\draw[->] (b2) -- node [above] {} (a3);
\draw[->] (b2) -- node [below] {} (a4);
\draw[->] (b3) -- node [above] {} (a5);
\draw[->] (b3) -- node [above] {} (a6);
\draw[->] (b4) -- node [above] {} (a7);
\draw[->] (b4) -- node [below] {} (a8);
\draw[->] (b5) -- node [above] {\ls{$s_{10}$}} (a9);
\draw[->] (b5) -- node [below] {\ls{$s_{11}$}} (a10);
\draw[->] (b6) -- node [above] {} (a11);
\draw[->] (b6) -- node [above] {} (a12);
\draw[->] (b7) -- node [above] {} (a13);
\draw[->] (b7) -- node [above] {} (a14);
\draw[->] (b8) -- node [above] {\ls{$s_{12}$}} (a15);
\draw[->] (b8) -- node [below] {\ls{$s_{13}$}} (a16);


\node [right, xshift=5] at (a1) {$p_{000}$};
\node [right, xshift=5] at (a2) {$p_{001}$};
\node [right, xshift=5] at (a3) {$p_{010}$};
\node [right, xshift=5] at (a4) {$p_{011}$};
\node [right, xshift=5] at (a5) {$p_{100}$};
\node [right, xshift=5] at (a6) {$p_{101}$};
\node [right, xshift=5] at (a7) {$p_{110}$};
\node [right, xshift=5] at (a8) {$p_{111}$};
\node [right, xshift=5] at (a9) {$p_{200}$};
\node [right, xshift=5] at (a10){$p_{201}$};
\node [right, xshift=5] at (a11){$p_{210}$};
\node [right, xshift=5] at (a12){$p_{211}$};
\node [right, xshift=5] at (a13){$p_{300}$};
\node [right, xshift=5] at (a14){$p_{301}$};
\node [right, xshift=5] at (a15){$p_{310}$};
\node [right, xshift=5] at (a16){$p_{311}$};

\end{tikzpicture} \,
\begin{tikzpicture}
\renewcommand{\xx}{1}
\renewcommand{\yy}{0.5}
\node at (2*\xx,14*\yy) {$\T_3:$};

\node (r) at (2*\xx,7.5*\yy) {\bt};

\node (c1) at (3*\xx,14.3*\yy) {\small$0$};
\node (c2) at (3*\xx,10.3*\yy) {\small$1$};
\node (c3) at (3*\xx,6.3*\yy) {\small$2$};
\node (c4) at (3*\xx,2.3*\yy) {\small$3$};

\node (c1) at (3*\xx,13.5*\yy) {\stage{Tan}{$\phantom{;}$}};
\node (c2) at (3*\xx,9.5*\yy) {\stage{Tan}{$\phantom{;}$}};
\node (c3) at (3*\xx,5.5*\yy) {\stage{VioletRed}{$\phantom{;}$}};
\node (c4) at (3*\xx,1.5*\yy) {\stage{VioletRed}{$\phantom{;}$}};

\node (b1) at (4*\xx,15.3*\yy) {\small $00$};
\node (b2) at (4*\xx,13.3*\yy) {\small $01$};
\node (b3) at (4*\xx,11.3*\yy) {\small $10$};
\node (b4) at (4*\xx,9.3*\yy) {\small $11$};
\node (b5) at (4*\xx,7.3*\yy) {\small $20$};
\node (b6) at (4*\xx,5.3*\yy) {\small $21$};
\node (b7) at (4*\xx,3.3*\yy) {\small $30$};
\node (b8) at (4*\xx,1.3*\yy) {\small $31$};

\node (b1) at (4*\xx,14.5*\yy) {\stage{Cyan}{$\phantom{;}$}};
\node (b2) at (4*\xx,12.5*\yy) {\stage{Cyan}{$\phantom{;}$}};
\node (b3) at (4*\xx,10.5*\yy) {\stage{Cyan}{$\phantom{;}$}};
\node (b4) at (4*\xx,8.5*\yy) {\stage{Gray}{$\phantom{;}$}};
\node (b5) at (4*\xx,6.5*\yy) {\stage{Gray}{$\phantom{;}$}};
\node (b6) at (4*\xx,4.5*\yy) {\stage{Gray}{$\phantom{;}$}};
\node (b7) at (4*\xx,2.5*\yy) {\stage{Gray}{$\phantom{;}$}};
\node (b8) at (4*\xx,0.5*\yy) {\stage{Cyan}{$\phantom{;}$}};

\node (a1) at (5*\xx,15*\yy) {\bt};
\node (a2) at (5*\xx,14*\yy) {\bt};
\node (a3) at (5*\xx,13*\yy) {\bt};
\node (a4) at (5*\xx,12*\yy) {\bt};
\node (a5) at (5*\xx,11*\yy) {\bt};
\node (a6) at (5*\xx,10*\yy) {\bt};
\node (a7) at (5*\xx,9*\yy) {\bt};
\node (a8) at (5*\xx,8*\yy) {\bt};
\node (a9) at (5*\xx,7*\yy) {\bt};
\node (a10) at (5*\xx,6*\yy) {\bt};
\node (a11) at (5*\xx,5*\yy) {\bt};
\node (a12) at (5*\xx,4*\yy) {\bt};
\node (a13) at (5*\xx,3*\yy) {\bt};
\node (a14) at (5*\xx,2*\yy) {\bt};
\node (a15) at (5*\xx,1*\yy) {\bt};
\node (a16) at (5*\xx,0*\yy) {\bt};

\draw[->] (r) -- node [above] {\ls{$s_0$}} (c1);
\draw[->] (r) -- node [below] {\ls{$s_1$}} (c2);
\draw[->] (r) -- node [above] {\ls{$s_2$}} (c3);
\draw[->] (r) -- node [below] {\ls{$s_3$}} (c4);

\draw[->] (c1) -- node [above] {\ls{$s_4$}} (b1);
\draw[->] (c1) -- node [below] {\ls{$s_5$}} (b2);
\draw[->] (c2) -- node [above] {} (b3);
\draw[->] (c2) -- node [below] {} (b4);
\draw[->] (c3) -- node [above] {\ls{$s_6$}} (b5);
\draw[->] (c3) -- node [below] {\ls{$s_7$}} (b6);
\draw[->] (c4) -- node [above] {} (b7);
\draw[->] (c4) -- node [below] {} (b8);

\draw[->] (b1) -- node [above] {\ls{$s_{8}$}} (a1);
\draw[->] (b1) -- node [below] {\ls{$s_{9}$}} (a2);
\draw[->] (b2) -- node [above] {} (a3);
\draw[->] (b2) -- node [below] {} (a4);
\draw[->] (b3) -- node [above] {} (a5);
\draw[->] (b3) -- node [above] {} (a6);
\draw[->] (b4) -- node [above] {} (a7);
\draw[->] (b4) -- node [below] {} (a8);
\draw[->] (b5) -- node [above] {\ls{$s_{10}$}} (a9);
\draw[->] (b5) -- node [below] {\ls{$s_{11}$}} (a10);
\draw[->] (b6) -- node [above] {} (a11);
\draw[->] (b6) -- node [above] {} (a12);
\draw[->] (b7) -- node [above] {} (a13);
\draw[->] (b7) -- node [above] {} (a14);
\draw[->] (b8) -- node [above] {} (a15);
\draw[->] (b8) -- node [below] {} (a16);


\node [right, xshift=5] at (a1) {$p_{000}$};
\node [right, xshift=5] at (a2) {$p_{001}$};
\node [right, xshift=5] at (a3) {$p_{010}$};
\node [right, xshift=5] at (a4) {$p_{011}$};
\node [right, xshift=5] at (a5) {$p_{100}$};
\node [right, xshift=5] at (a6) {$p_{101}$};
\node [right, xshift=5] at (a7) {$p_{110}$};
\node [right, xshift=5] at (a8) {$p_{111}$};
\node [right, xshift=5] at (a9) {$p_{200}$};
\node [right, xshift=5] at (a10){$p_{201}$};
\node [right, xshift=5] at (a11){$p_{210}$};
\node [right, xshift=5] at (a12){$p_{211}$};
\node [right, xshift=5] at (a13){$p_{300}$};
\node [right, xshift=5] at (a14){$p_{301}$};
\node [right, xshift=5] at (a15){$p_{310}$};
\node [right, xshift=5] at (a16){$p_{311}$};

\end{tikzpicture}

\end{center} 
\caption{Three examples of staged trees. In each tree two vertices with the same color are in the same
stage.}\label{fig:trees}
\end{figure}

We are interested in relating the combinatorial properties of the staged tree $(\T,\theta)$ with the properties
of the toric ideal $\ker(\varphi_{\T})$. Before we dive into the combinatorics of staged
trees we present our main Theorem~\ref{thm:main}. In its statement we use the notion of \emph{balanced} staged tree  and of \emph{stratified} staged tree. 
\begin{theorem}\label{thm:main}
If $(\T,\theta)$ is a balanced and stratified staged tree then $\ker(\varphi_{\T})$ is generated by a 
quadratic Gr\"obner basis with squarefree initial ideal. 
\end{theorem}
We clarify that the conditions of $(\T,\theta)$ being balanced and stratified
in Theorem~\ref{thm:main} are sufficient for $\ker(\varphi_{\T})$ to have a quadratic Gr\"obner basis
but are not necessary.
In the examples of staged trees in Figure~\ref{fig:trees}, all of the trees $\T_1,\T_2,\T_3$
are stratified but only $\T_1$ is balanced. Even though $\T_2$ is not balanced, it has
a quadratic Gr\"obner basis with squarefree initial terms.

\begin{definition}
Let $\T$ be a tree. For $v\in V$, the \emph{level} of $v$ is denoted by $\ell(v)$ and it 
is equal to the number of
edges in the unique path from the root of $\T$ to $v$. If all the leaves in $\T$ have the
same level then the level of $\T$ is equal to the level of any of its leaves. The staged tree $(\T,\theta)$ is \emph{stratified} if
all its leaves have the same level and if every two vertices in the same stage have the same level.
\end{definition}

It is easy to check that the trees in Figure~\ref{fig:trees} are stratified. Namely we only need
to verify that every two vertices with the same color are also in the same level. Notice that the
combinatorial condition of
$(\T,\theta)$ being stratified translates into the algebraic condition that the map $\varphi_{\T}$ is
squarefree.

We now turn our attention to the definition of a balanced staged tree. 
This definition is formulated
in terms of polynomials associated to each vertex of the tree. We proceed to explain their notation
and basic properties. 
\begin{definition}\label{def:inter}
Let $(\T,\theta)$ be a staged tree, $v \in V$, and  $\T_{v}$ the subtree of $\T$ rooted at $v$. The tree
 $\T_{v}$ is a staged tree with the induced labelling from $\T$.
 Let $\Lambda_{v} $ be the set of $v$-to-leaf paths in $\T$ and define
\[t(v):=\sum_{\lambda\in \Lambda_{v}} \prod_{e\in E(\lambda)} \theta(e). \] 
When $v$ is the root of $\T$, the polynomial $t(v)$ is called the \emph{interpolating polynomial} of
$\T$. Two staged trees $(\T,\theta)$ and $(\T,\theta')$
with the same label set $\la$ are \emph{polynomially equivalent} if their interpolating polynomials are equal. 
\end{definition}

The interpolating polynomial of a staged tree is an important tool in the study of the statistical 
properties of staged tree models. 
 This polynomial was defined by G\"orgen and
Smith in \cite{GSmith} and further studied by G\"orgen et al. in \cite{Gdiscovery}. Although these two articles
are written for a statistical audience, we would like to emphasize that their symbolic algebra
approach to the study of statistical models proves to be very important for the use of these models
in practice. We will define the statistical model associated to a staged tree and connect Theorem~\ref{thm:main} to other results in Algebraic Statistics in Section~\ref{sec:stats}.

If $(\T,\theta)$ is a staged tree, the polynomials $t(\cdot)$ satisfy a recursive relation. This
relation is useful to prove statements about the algebraic and combinatorial properties of
$\T$. We state this property as a lemma.
\begin{lemma}[{\cite[Theorem 1]{Gdiscovery} }]\label{lem:recursive}
Let $(\T,\theta)$ be a staged tree, $v\in V$ and $\mathrm{ch}(v)=\{v_0,\ldots,v_k\}$. Then the polynomial
$t(v)$ admits the recursive representation 
$t(v)=\sum_{i=0}^k\theta(v,v_i)t(v_i).$
\end{lemma}
\begin{example}
Consider the staged tree $\T_1$ in Figure~\ref{fig:trees}. If $v$ and $w$ are orange and blue vertices in
$\T_1$ respectively and $r$ is the root of $\T_1$ then $t(v)=s_6(s_{10}+s_{11})+s_7(s_{12}+s_{13})$,
$t(w)=s_8(s_{10}+s_{11})+s_9(s_{12}+s_{13})$, and $t(r)=(s_0s_2+s_1s_4)t(v)+(s_0s_3+s_1s_5)t(w)$.

\end{example}

\begin{definition}\label{def:star}
Let $(\T,\theta)$ be a staged tree and $v,w$ be two
vertices in the same stage with  $\mathrm{ch}(v)=\{v_0,\ldots, v_k\}$ and
 $\mathrm{ch}(w)=\{w_0,\ldots,w_k\}$. 
After a possible  reindexing of the elements in $\mathrm{ch}(w)$
we may assume that $\theta(v,v_i)=\theta(w,w_i)$ for all $i \in \{0,\ldots, k\}$. The vertices $v,w$ satisfy condition $(\star)$ if
\begin{align} \label{eq:star}
t(v_i)t( w_j)=t(w_i)t(v_j) \text{ in } \R[\Theta]_{\T}, \text{ for all } i\neq j \in  \{0,\ldots, k\} .
\end{align}
The staged tree $(\T,\theta)$ is \emph{balanced} if every pair of vertices  in the same stage  satisfy condition $(\star)$.
\end{definition}

\begin{example}
The two staged trees $\T_2,\T_3$ in Figure~\ref{fig:trees} are not balanced. The two pink vertices in $\T_2$
do not satisfy condition $(\star)$ because $(s_{10}+s_{11})(s_{12}+s_{13})\neq (s_{10}+s_{11})^2$.
By a similar argument we can check that $\T_3$ is also not balanced.
\end{example}

Although condition $(\star)$ seems to be algebraic and hard to check, in many cases it is very combinatorial. To formulate a precise instance where this is true we need the following definition.

\begin{definition}
Let $(\T,\theta)$ be a staged tree. We say that two vertices $v,w\in V$ are in the same \emph{position} if
they are in the same stage and
$t(v)=t(w)$. 
\end{definition}

The notion of position for vertices in the same stage
was formulated in \cite{SmithAnderson}. Intuitively it means that if we regard
the subtrees $\T_v$ and $\T_w$ as representing the unfolding of a sequence of events, then the future
of $v$ and $w$ is essentially  the same.  In the next lemma we use positions of vertices to provide 
a sufficient condition on a stratified staged tree $(\T,\theta)$ to be balanced.

\begin{lemma} \label{lem:positions}
Let $(\T,\theta)$ be a stratified staged tree. Suppose that every two vertices in $\T$ that are in the same stage
are also in the same position. Then $(\T,\theta)$ is balanced.
\end{lemma}
\begin{proof}
Following Definition~\ref{def:star} it suffices to prove that any two vertices
in the same position satisfy condition $(\star)$. Let $v,w$ be two vertices in the same position,
by definition this means that they are in the same stage and $t(v)=t(w)$. Let $\mathrm{ch}(v)=\{v_0,\ldots,v_k\}$ and $\mathrm{ch}(w)=\{w_0,\ldots,w_k\}$, after possibly permuting the subindices
of elements in $\mathrm{ch}(w)$, we may assume $\theta(v,v_i)=\theta(w,w_i)$. Using Lemma~\ref{lem:recursive}
 we write
\begin{align*}
t(v)=t(w) \ifi \sum_{i=0}^k\theta(v,v_i)t(v_i)=\sum_{i=0}^k\theta(w,w_i)t(w_i) 
           \ifi\sum_{i=0}^k\theta(v,v_i)(t(v_i)-t(w_i))=0.
\end{align*}
Since $(\T,\theta)$ is stratified, the variables appearing in the polynomials $t(v_i),t(w_i)$
are disjoint from the set of variables $\{\theta(v,v_0),\ldots,\theta(v,v_k)\}$. Thus
$t(v_i)=t(w_i)$ for all $i\in \{0,\ldots,k\}$.
 It follows that for all $i,j\in\{0,\ldots,k\}$ the equality $t(v_i)t( w_j)=t(w_i)t(v_j) $ is true.
 Hence $(\T,\theta)$ is balanced.

\end{proof}

\begin{example}
The staged tree $\T_1$ in Figure~\ref{fig:trees} is balanced. This can be readily checked by noting that
the blue vertices are in the same position
and that the same is true for the orange vertices. The tree $\T$ in Figure~\ref{fig:asymtrees}
is an example of a balanced staged tree in which the blue vertices are not in the same position.
\end{example}

\section{Toric fiber products for staged trees} \label{sec:tfps}
In this section we define a toric fiber product for staged trees. We then use these results
in Section~\ref{sec:proofs} to
prove Theorem~\ref{thm:main}. We start with a review of toric fiber products following \cite{Sullivant2007}.



 Given a positive 
integer $m$, set $[m]=\{1,2,\ldots,m\}$. Let $r$ be a positive integer and let
$s$ and $t$ be two vectors of positive integers in $\Z^r_{>0}$.  Consider the homogeneous,
multigraded polynomial rings
\[
\mathbb{K}[x] := \mathbb{K}[x_j^i \mid i\in [r], j\in [s_i]  ]\quad
\text{and}\quad \mathbb{K}[y] := \mathbb{K}[y_k^i \mid i\in [r], k\in [t_i]  ]
\]
with the same multigrading
\[
\degree(x_j^i) = \degree(y_k^i) = \bm{a}^i \in \Z^d.
\]
Denote by $\mathcal{A} = \{ \bm{a}^1,\dots,\bm{a}^r \}$ the set of all
multidegrees of these variables and assume that there exists a vector
$w\in \mathbb{Q}^d$ such that $\<w,\bm{a}^i\>=1$ for any $\bm{a}^i\in
\mathcal{A}.$ If $I\subseteq \mathbb{K}[x]$ and $J\subseteq
\mathbb{K}[y]$ are homogeneous ideals, then the quotient rings $R =
\mathbb{K}[x]/I$ and $S=\mathbb{K}[y]/J$ are also multigraded
rings. Let
\[
\mathbb{K}[z] := \mathbb{K}[z_{jk}^i \mid i\in [r], j\in [s_i], k\in [t_i]]
\]
and consider the ring homomorphism
\begin{align*}
\phi_{I,J} : \mathbb{K}[z] &\rightarrow R \otimes_{\mathbb{K}} S \\
z_{jk}^i &\mapsto \overline{x_j^i} \otimes
\overline{y_k^i},
\end{align*}
where $\overline{x_j^i}$ and $\overline{y_k^i}$ are the equivalence
classes of $x_j^i$ and $y_k^i$ respectively.

\begin{definition}
  The \textit{toric fiber product} of $I$ and $J$ is
  $I\times_{\mathcal{A}} J := \mathrm{ker}( \phi_{I,J}).$ 
\end{definition}

Let $(\T,\theta)$ be a staged tree with root $r$.  We recursively define an indexing of the vertices in $V\setminus \{r\}$.
This identifies each vertex in $V\setminus\{r\}$ with a unique index. From this point on we refer to
an element  in $V$ via its index $\bm{a}$ or by $r$ in case the vertex is the root. The children of $r$ are indexed by $\{0,1,\dots,k\}.$  For
 $\bm{a}\in V\setminus\{r\}$, we index the children of
$\bm{a}$ as follows. If $E(\bm{a}) = \emptyset$ then
$\bm{a}$ is a leaf of the tree and we are done. If
$|E(\bm{a})| = j+1,$ index the children of $\bm{a}$
by $\bm{a}0,...,\bm{a}j.$ This way each vertex
in $V$  is indexed by a finite sequence of nonnegative integers
\[
\bm{a} = a_1a_2\cdots a_{\ell},
\]
where $\ell$ is the level of $\bm{a}$.    All vertices of the trees in Figure~\ref{fig:trees} are
indexed following this rule. In Figure~\ref{fig:trees} the index of each vertex is displayed immediately above each vertex
and on the side for the leaves. We denote by $\bm{i}_{\T}$ the set of
indices of the leaves in $\T$.

\begin{definition}\label{def:onelevel}
Let $(\T,\theta)$ be a staged tree. If every leaf in $\T$ has level one then we call $\T$ a 
\emph{one-level tree}. We reserve for it the special notation $(\mathcal{B},t)$ where $t$
is the set of edge labels of $\mathcal{B}$.
\end{definition}

\begin{definition} \label{def:tnew}
Let $(\T,\theta)$ be a staged tree and $G=\{G_1,\ldots,G_r\}$ be a partition of the set of leaves $\bm{i}_{\T}$.
For each $i\in[r]$, let $(\B_{i},t^{(i)})$ be a one-level tree as in Definition~\ref{def:onelevel} with label set $t^{(i)}$. We
define the \emph{gluing component} $\T_{G}$ associated to $\T$ and $G$ as the disjoint union of $(\B_i,t^{(i)})$, namely 
\[
\T_G = \ds \bigsqcup_{i\in [r]} (\B_i,t^{(i)}).
\]
We denote by $[\T,\T_G]$ the tree obtained by gluing $\B_i$ to the leaf $\bm{a}$ for all $\bm{a}\in
G_i$ and all $i\in [r]$. 
\end{definition}

\begin{remark}
The labelling in $[\T,\T_{G}]$  is inherited from the labellings of $\T$ and $\T_G$. With this labelling  $[\T,\T_G]$ is a staged tree. Moreover,  $\bm{i}_{[\T,\T_{G}]}=
\{\bm{a}k\,|\, \bm{a}\in G_i, k\in \bm{i}_{\B_i}, i\in [r]\}$. The stages in $[\T,\T_{G}]$ are the ones inherited from
$\T$ union the new stages determined by $G$. This means that two vertices $\bm{a},\bm{b}\in \bm{i}_{\T}$
are in the same stage in $[\T,\T_G]$ provided $\bm{a},\bm{b}\in G_i$. 
\end{remark}

We relate $\ker(\varphi_{[\T,\T_G]})$ to the toric fiber product of the two ideals $\ker(\varphi_{\T})$
and the zero ideal $\langle 0\rangle$.
Let $\T,G,\T_G$ and $[\T,\T_G]$ be as in Definition~\ref{def:tnew}. First we associate to $\T_G$
the ring  $\R[p]_{\T_G}:=\R[p_k^i\,\mid\,  k\in \bm{i}_{\B_i}, i\in [r]]$ and the ring map
\begin{align*}
  \varphi_{\T_G}: \R[p]_{\T_G} &\rightarrow \R[\Theta]_{\T_G} \\ p_{k}^i
  &\mapsto t_k^{(i)}.
\end{align*}
Since there is a one-to-one correspondence between the variables $p_k^i$ and $t_k^{(i)}$, we see that $\varphi_{\T_G}$
is an isomorphism. In particular, $\ker(\varphi_{\T_G})=\langle 0\rangle $.  Second,  using $G$ we  regroup the variables in
$\R[p]_{\T}$ by
\[
\R[p]_{\T} = \R[p_{\bm{j}}^i \mid \bm{j}\in G_i,\, i\in
  [r]].
\]

We define the multigrading on the polynomial rings $\R[p]_{\T}$ and
$\R[p]_{\T_G}$ as
\[
 \degree(p_{\bm{j}}^i) = \degree(p_{k}^i) = \bm{e}_i,\, \text{ for 
 } \, \bm{j}\in G_i,\, k\in \bm{i}_{\mathcal{B}_{i}}, \, i\in [r].
\]
Here $\bm{e}_i$ is the $i$th standard unit vector in $\Z^{r}$. If $\mathcal{A}$ is the set of all these
multidegrees, then $\mathcal{A}$ is linearly independent as it is a
collection of standard unit vectors in $\Z^r$. We say a homogeneous polynomial in
the ring $\R[p]_{\T}$ or $\R[p]_{\T_G}$ is $\A$-\emph{graded} whenever it is homogeneous with respect
to the multigrading determined by $\A$.

Fix  $R= \R[p]_{\T}/\ker(\varphi_{\T})$, $S=\R[p]_{\T_G}/\ker(\varphi_{\T_G})$ and let 
$\R[p]_{[\T,\T_G]}=\R[p_{\bm{j}k}^i \mid \, \bm{j}\in G_i,\, k\in \bm{i}_{\mathcal{B}_{i}},\, i\in [r]]$.
Consider the ring homomorphism
\begin{align}\label{eq:varsforTring}
\psi : \R[p]_{[\T,\T_G]} &\rightarrow R
\otimes_{\R} S \nonumber \\ p_{\bm{j}k}^i &\mapsto \overline{p_{\bm{j}}^i} \otimes
\overline{p_{k}^i},\quad \text{for }  \bm{j}\in G_i,\, k\in\bm{i}_{\B_i}\, i\in [r].
\end{align}
The ideal $\ker (\psi)=\ker(\varphi_{\T}) \times_{\mathcal{A}}\langle 0 \rangle$ is the
toric fiber product of $\ker(\varphi_{\T})$ and $\langle 0 \rangle$.
 \begin{proposition} \label{lem:tfp}
     Let $\T, G, \T_G$ and $[\T,\T_G]$ be as in Definition~\ref{def:tnew}. 
     Suppose that $\ker(\varphi_{\T})$ is homogeneous with respect to the multigrading given 
     by $\mathcal{A}$. Then 
     \[ \ker(\varphi_{[\T,\T_G]})= \ker(\varphi_{\T}) \times_{\mathcal{A}}\langle 0 \rangle.\] 
     \end{proposition}
      \begin{proof}
        For each $\bm{j}k\in \bm{i}_{[\T,\T_G]}$, the map $\varphi_{[\T,\T_G]}$ can be rewritten as 
        \[p_{\bm{j}k}\mapsto 
        z\cdot \prod_{e\in E( \lambda_{\bm{j}k})}\theta(e) = z\cdot\left(\prod_{e\in E(\lambda_{\bm{j}})}\theta(e)\right)
        \theta({\bm{j}},\bm{j}k)\]
         where $\theta({\bm{j}},\bm{j}k)=t_k^{(i)}$. Written in this factored form we see that $\varphi_{[\T,\T_G]}$ factors trough $\psi$. 
        Since $\ker(\varphi_{\T})$ is $\mathcal{A}$-homogeneous we conclude $\ker(\varphi_{[\T,\T_G]})=\ker(\varphi_{\T}) \times_{\mathcal{A}}\langle 0\rangle$. 

 \end{proof}
We make several remarks on the scope of Proposition~\ref{lem:tfp} via the next set of examples.
\begin{definition} \label{def:sublevel}
Let $(\T,\theta)$ be a stratified staged tree of level $m$. For $1\leq q\leq
m$ we define $V_{\leq q}:=\cup_{i=0}^q V_i$ where $V_q:=\{v\in V\,|\, \ell(v)= q\}$ and
$E_{\leq q}:=\{(v,w)\in E \,|\,
\ell(v)\leq q ,\ell(w)\leq q \}$. The staged tree $(\T^{(q)},\theta)$ is defined to be the subtree $\T^{(q)}$ 
of $\T$ with vertex set
$V_{\leq q}$, edge set $E_{\leq q}$ and labeling induced from $\T$.
\end{definition}
\begin{example} \label{ex:eqnst1}
Consider the staged tree $\T_1$ in Figure~\ref{fig:trees} and  let $\T=\T_1^{(3)}$ as in Definition~\ref{def:sublevel}. Then $\T$ is a staged tree with label set $\{s_0,\ldots,s_9\}$. 
Fix $G=\{\{000,010,100,110\},\{001,011,101,111\}\}$ 
and $\T_{G}=(\B_1,\{s_{10},s_{11}\})\sqcup(\B_2,\{s_{12},s_{13}\})$. With this choice
of $\T,G$ and $\T_G$ we see that $\T_1=[\T,\T_G]$.
Now 
$\R[p]_{\T}=\R[p_{\bm{a}}^i\,|\,  \bm{a}\in G_i, i\in\{1,2\}]$, hence
$\deg(p_{000}^1,p_{010}^1,p_{100}^1,p_{110}^1)=\bm{e}_1$ and $\deg(p_{001}^2,p_{011}^2,p_{101}^2,p_{111}^2)=
\bm{e}_2$ so $\A=\{\bm{e}_1,\bm{e_2}\}\subset \Z^2$ is of full rank.  The ideal
 \[\ker(\varphi_{\T})=\langle 
p_{000}^1p_{101}^2-p_{100}^1p_{011}^2,p_{010}^1p_{111}^2-p_{110}^1p_{011}^2\rangle\]
is $\mathcal{A}$-graded. Hence
by Proposition~\ref{lem:tfp} $\ker(\varphi_{\T_1})=\ker(\varphi_{\T})\times_{\mathcal{A}}\langle
0 \rangle$.
\end{example}

\begin{example}
Let $\T_2$ be the staged tree from Figure~\ref{fig:trees}. We proceed in a similar fashion as
in Example~\ref{ex:eqnst1}. Set $\T=\T_{2}^{(2)}$,  $G=\{\{00,01,10\},
\{11,31\},\{20,21,30\}\}$ and  $\T_{G}=(\B_1,\{s_{8},s_{9}\})\sqcup(\B_2,\{s_{12},s_{13}\})
\sqcup (\B_3,\{s_{10},s_{11}\})$. Then $\T_2=[\T,\T_G]$. The set $G$ defines a multigrading on $\R[p]_{\T}$
with $\mathcal{A}=\{\bm{e}_1,\bm{e}_2,\bm{e}_3\}\subset \Z^3$.
The ideal $\ker(\varphi_{\T})=\langle 
p_{00}^1p_{11}^2-p_{10}^1p_{01}^1,p_{20}^3p_{31}^2-p_{30}^3p_{21}^3\rangle$ is not $\A$-graded. Thus in this case 
$\ker(\varphi_{\T_2})\neq \ker(\varphi_{\T})\times_{\mathcal{A}}
\langle 0\rangle$.
\end{example}

\begin{example}
Let $\T_3$ be the staged tree from Figure~\ref{fig:trees} and $\T=\T_3^{(2)}$. As in the previous examples,
fix  $G=\{\{00,01,10,31\},\{20,21,30,11\}\}$ and $\T_{G}=(\B_1,\{s_8,s_9\})\sqcup (\B_2,
\{s_{10},s_{11}\})$ so $\T_3=[\T,\T_G]$. The set $G$ defines a multigrading on $\R[p]_{\T}$
with $\mathcal{A}=\{\bm{e}_1,\bm{e}_2\}\subset \Z^2$. The ideal 
$\ker(\varphi_{\T})=\langle 
p_{00}^1p_{11}^2-p_{10}^1p_{01}^1,p_{20}^2p_{31}^1-p_{30}^2p_{21}^2\rangle$ is
not $\A$-graded.
However there is a nonempty principal subideal of $\ker(\varphi_{\T})$ that
is $\mathcal{A}$-homogeneous. This principal ideal $Q$ is generated by the quartic 
$p_{00}^1p_{11}^2p_{20}^2p_{31}^1-p_{01}^1p_{10}^1p_{21}^2p_{30}^2$. 
In this case $\ker(\varphi_{\T_3})=Q\times_{\mathcal{A}}\langle 0 \rangle$. This does not fall in the context of Proposition~\ref{lem:tfp} since $\ker(\varphi_{\T_3})\neq \ker(\varphi_{\T})\times_{\mathcal{A}} \langle 0 \rangle$.
\end{example}
\section{Proof of main theorem} \label{sec:proofs}
The main ingredient in the proof of Theorem~\ref{thm:main}
is the toric fiber product. The idea of the proof is that
when $(\T,\theta)$ is balanced and stratified, we can construct $\ker(\varphi_{\T})$ in a finite number
of steps using Proposition~\ref{lem:tfp}.

Start with a one-level probability tree $\T_1$. Let $G^1$ be a partition of $\bm{i}_{\T_1}$,
$\T_{G^1}$ be a gluing component and set $\T_2=[\T_1,\T_{G^1}]$. In the inductive step,
$\T_{j+1}=[\T_j,\T_{G^j}]$ with $r_j:=|G^j|$. At
each step $j$ we also require that the label set of $\T_{G^j}$
is disjoint from the label set of $\T_j$. After $n$ iterations, we
obtain a stratified staged tree $\T_n$ whose set of stages is exactly $\cup_{j=1}^{n-1} G^j$.
Whenever a staged tree $(\T,\theta)$ is constructed in this way so $\T=\T_n$ for some $n$ we say $\T$ is an \emph{inductively 
constructed}
staged tree.

We recall a theorem from \cite{Sullivant2007} concerning the defining equations of a toric fiber product.
To state this theorem we use the notation for toric fiber products from Section~\ref{sec:tfps}.
\begin{theorem}[{\cite[Theorem~2.9]{Sullivant2007}}]\label{thm:tfp}
Suppose that $\mathcal{A}$ is linearly independent. Let $F\subset I$ be a homogeneous
Gr\"{o}bner basis for $I$ with respect to the weight vector $\omega_1$ and let $H\subset J$
be a homogeneous Gr\"{o}bner basis for $J$ with respect to the weight vector $\omega_2$. Let
$\omega$ be a weight vector such that $\Quad_B$ is a Gr\"{o}bner basis for $I_{B}$.
Then \[\Lift(F)\cup \Lift(H)\cup \Quad_B\]
is a Gr\"{o}bner basis for $I\times_{\mathcal{A}}J$ with respect to the weight order $\phi_{B}^{\ast}(
\omega_1,\omega_2)+ \epsilon \omega$ for sufficiently small $\epsilon>0$.
\end{theorem}
We use this result to obtain a Gr\"obner basis of any inductively constructed staged tree. Theorem~\ref{thm:tfp}  has two important ingredients. The first is the set of equations denoted
by $\Quad_{B}$, these are quadratic equations that emerge from the construction of the toric fiber product. The second ingredient
is the set $\Lift(F)\cup \Lift(H)$, these are the lifts of generators of the ideals $I$ and $J$ respectively.
 
We start by explaining the construction of the elements in $\Quad_B$. Let $\T_j$ be an inductively constructed staged tree and $\T_{j+1}=[\T_j,\T_{G^j}]$ with
$r_j=|G^j|$. Consider the monomial map
\begin{align} \label{monomParam}
  \phi_{B_j} : \R[p]_{\T_{j+1}} &\rightarrow \R[p_{\bm{a}}^i,p_k^i \mid 
    \bm{a}\in G_i^j, k\in \bm{i}_{\B_i^j},\, i\in [r_j]] \nonumber \\ p_{\bm{a}k}^i
  &\mapsto p_{\bm{a}}^i p_k^i
\end{align}
where $B_j$ denotes the exponent matrix of the monomial map $\phi_{B_j}$. Set $I_{B_j}=
\ker(\phi_{B_j})$ and
\[
\Quad_{B_j} = \{
\underline{p^i_{\bm{a}k_1}p^i_{\bm{b}k_2}}-p^i_{\bm{b}k_1}p^i_{\bm{a}k_2}
\mid  \bm{a},\bm{b}\in G_i^j, k_1\neq k_2\in \bm{i}_{\B_i^j}, i\in [r_j] \}.
\] By Proposition~10 in
\cite{Sullivant2007} $\Quad_{B_j}$ is a Gr\"{o}bner basis for $I_{B_j}$ with respect to
any term order that selects the underlined terms as leading terms.

We now explain the construction of the elements in $\Lift(F)\cup \Lift(H)$. Fix $\T, G, \T_G$ and $[\T,\T_G]$ as in Definition~\ref{def:tnew} and denote by $\mathcal{A}$ the multigrading of the rings
 $\R[p]_{\T},\R[p]_{[\T,\T_G]}$ defined by $G$. We recall the definition of a lift of an $\mathcal{A}$-graded polynomial in $\R[p]_{\T}$ to the polynomial ring
$\R[p]_{[\T,\T_G]}$ of the toric fiber product. Since we only consider pure quadratic binomials, we 
restrict the definition from \cite{Sullivant2007} of lift to this particular case. 
Consider the $\A$-graded polynomial \[f=p_{\bm{a}_1}^{i_1}p_{\bm{a}_2}^{i_2}-p_{\bm{a}_3}^{i_1}p_{\bm{a}_4}^{i_2}\in \R[p]_{\T},\]
where  $
    \bm{a}_1,\bm{a}_3\in G_{i_1}$, $ \bm{a}_2,\bm{a}_4\in G_{i_2}$ and $ i_1,i_2\in[r]$. Set $k=(k_1,k_2)$
    with $k_1\in \bm{i}_{\B_{i_1}}, k_2\in \bm{i}_{\B_{i_2}}$ and consider $f_k\in \R[p]_{[\T,\T_G]}$ defined
    by \[f_k=  p_{\bm{a}_1k_1}^{i_1}p_{\bm{a}_2k_2}^{i_2}-p_{\bm{a}_3k_1}^{i_1}p_{\bm{a}_4k_2}^{i_2}.\]
\begin{definition}
Let $\mathcal{A}$ be the multigrading of the rings
 $\R[p]_{\T},\R[p]_{[\T,\T_G]}$ defined by $G$ and let $F \in \ker(\varphi_{\T})$ be a collection of pure $\A$-graded binomials.
We associate to each $f\in F$ the set $T_f=\bm{i}_{\B_{i_1}}\times\bm{i}_{\B_{i_2}}$ of indices and define
\[\Lift(F)=\{f_k: f\in F, k\in T_f\}.\]
The set $\Lift(F)$ is called the lifting of $F$ to $\ker(\varphi_{\T})\times_{\A} \langle 0 \rangle$ (see \cite{Sullivant2007}). 
\end{definition}

\begin{definition} Let $\T_n$ be an inductively constructed staged tree with $\T_i=[\T_{i-1},\T_{G^i}]$
for $1\leq i\leq n$ and let $\A_i$ be the grading in $\R[p]_{\T_i}$ determined by $G^i$.
Fix two nonnegative integers $i,q$ with $0\leq i+q\leq n-1$ . We define
\begin{align*}
\Lift^{q}(\Quad_{B_i}):= \Lift_{\mathcal{A}_{i+q}}(\cdots(\Lift_{\mathcal{A}_{i+2}}(\Lift_{\mathcal{A}_{i+1}}(\Quad_{B_i})))\cdots)
\end{align*}
where the subscript in $\Lift_{\mathcal{A}}(\,\cdot\,)$ indicates the grading
of the argument is with respect to $\A$.
\end{definition}

We formulate a lemma that says that certain staged subtrees of balanced and stratified
staged trees are also balanced and stratified.
\begin{lemma} \label{lem:subbalanced}
Let $(\T,\theta)$ be a staged tree of level $m$ and let $q$ be a positive integer with $1\leq q \leq m-1$.
If $(\T,\theta)$ is balanced and stratified then the staged subtree $\T^{(q)}$ of $\T$ as defined in  
Definition~\ref{def:sublevel} is
also balanced and stratified.
\end{lemma}
\begin{proof}
We must prove that if $\bm{a},\bm{b}$ are two vertices in $\T^{(q)}$ that are in the same stage,
then they satisfy condition $(\star)$ in $\R[\Theta]_{\T^{(q)}}$. Since $\T$ is balanced, then $\aaa,\bbb$
satisfy condition $(\star)$ in $\R[\Theta]_{\T}$. Namely,
\begin{align} \label{eq:recursivestar}
t(\aaa k_1)t(\bbb k_2)=t(\aaa k_2)t(\bbb k_1) \text{ in } \R[\Theta]_{\T} \quad\text{ for all } 
k_1,k_2\in\{0,\ldots, |\mathrm{ch}(\aaa)|\}.
\end{align}
For a vertex $v$ in $\T^{(q)}$ define
$[v]=\{\bm{u}\in \bm{i}_{\T^{(q)}}\,|\, \text{the root-to-}\bm{u} \text{ path in } \T^{(q)} \text{ goes
through } v \}$. Then by a repeated use of Lemma~\ref{lem:recursive}, for $\bm{c}\in \{\aaa k_1,\bbb 
k_2,
\aaa k_2,\bbb k_1\}$,
\[t(\bm{c})=\sum_{\bm{u}\in [\bm{c}]} \prod_{e\in E(\lambda_{\bm{u}})}\theta(e)t(\bm{u})\]
where $\lambda_{\bm{u}}$ is the $\bm{c}$ to $\bm{u}$ path in $\T^{(q)}$. Here $t(\bm{c})$ is
an element of $\R[\Theta]_{\T}$. Denote by $t(\bm{c})|_{\T^{(q)}}$ the polynomial obtained
from $t(\bm{c})$ by specializing $t(\bm{u})=1$ for all $\bm{u}\in [\bm{c}]$.
Since $\T$ is stratified, $t(\bm{c})|_{\T^{(q)}}$ is the interpolating polynomial of the subtree
$\T^{(q)}$ rooted at $\bm{c}$. Applying this specialization to (\ref{eq:recursivestar})
yields condition $(\star)$ for the vertices $\aaa,\bbb$ in $\R[\Theta]_{\T^{(q)}}$.

\end{proof}
\begin{proposition} \label{prop:extensions}
Let  $\T_i$ be a balanced and inductively constructed staged tree. Suppose $\T_{i+1}=[\T_i,\T_{G^i}]$ 
 and $\T_{i+1}$ is balanced. Then the elements in  
\begin{align*}
\Lift^{i-2}(\Quad_{B_1}), \Lift^{i-3}(\Quad_{B_2}), \ldots, \Lift(\Quad_{B_{i-2}}),\Quad_{B_{i-1}} 
\end{align*} are $\mathcal{A}_i$-graded. 
\end{proposition}
\begin{proof}
Note that any inductively constructed staged tree is stratified, therefore $\T_i$ and $\T_{i+1}$
are stratified.
By assumption $\T_i$ is inductively constructed, therefore there is a sequence of trees and
gluing components $(\T_1,\T_{G^1}),\ldots,(\T_{i-1},\T_{G^{i-1}})$ from which $\T_i$ is constructed.
Fix $q\in \{0,1,\ldots, i-2\}$ and $j=i-q-1$, we  show that the binomials in $\Lift^q(\Quad_{B_j})$
are  \mbox{$\mathcal{A}_{i}$-graded}. To this end we
 prove that for $m$ such that $0\leq m\leq q$,
the equations in $\Lift^{m}(\Quad_{B_j})$ are $\mathcal{A}_{j+m+1}$-graded. The 
proof is by induction on $m$. 

Fix $m=0$,
 we show that the elements in $\Quad_{B_j}$ are $\mathcal{A}_{j+1}$-graded.
 The equations   
\[\Quad_{B_j}=\bigcup_{\alpha=1}^{r_j}\{p_{\bm{a}k_1}p_{\bm{b}k_2}-p_{\bm{b}k_1}p_{\bm{a}k_2}\mid \bm{a},\bm{b}\in 
G_{\alpha}^{j}, k_1,k_2 \in \bm{i}_{\mathcal{B}_{\alpha}^j}\}\]
 are the generators of $I_{B_j}$.
The multidegrees $\mathcal{A}_{j+1}$ are defined according to the partition $G^{j+1}$ of the leaves
of $\T_{j+1}$. If two leaves in $\T_{j+1}$ are in the same set of the partition $G^{j+1}$ then
they have the same degree. Since $\T_i$ is balanced and $\T_{j+2}=\T_{i}^{(j+2)}$ then by Lemma~\ref{lem:subbalanced} the staged tree $\T_{j+2}$ is balanced. Therefore 
all the vertices in the stages $G^j=\{G_1^j,\ldots,G_{r_j}^j\}$ satisfy condition $(\star)$ 
in $\R[\Theta]_{\T_{j+2}}$. This means that
for all
$\alpha \in \{1,\ldots, r_{j}\}$ and $\bm{a},\bm{b}\in G_{\alpha}^j$
\begin{align}\label{eq:star1}
t({\bm{a}k_1})t({\bm{b}k_2})=t({\bm{b}k_1})t({\bm{a}k_2})\text{ for }  k_1,k_2 \in \bm{i}_{\mathcal{B}_{\alpha}^j}
%
\end{align}
where $\mathrm{ch}(\bm{a})=\{\bm{a}k\,|\, k\in \bm{i}_{\B_{\alpha}^j}\}$ and 
$\mathrm{ch}(\bm{b})=\{\bm{b}k\,|\, k\in \bm{i}_{\B_{\alpha}^j}\}$.

Using the construction of $\T_{j+2}$ from $\T_{j+1}$ and $\T_{G^{j+1}}$, we see that for any index $\bm{c}\in\{\bm{a}k,\bm{b}k\mid k\in \bm{i}_{\mathcal{B}_{\alpha}^j}\}$ of the leaves of $\T_{j+1}$, $t(\bm{c})$ is equal to the interpolating polynomial $t_{\B^{j+1}}$ of some one-level probability tree
$\mathcal{B}^{j+1}$  in
$\T_{G^{j+1}}$. Also, the assignment $t(\bm{c})=t_{\B_{\*}^{j+1}}$ determines the degree of $p_{\bm{c}}$
in $\R[p]_{\T_{j+1}}$. To be precise, if $t(\bm{c})=t(\bm{c}')$ for $\bm{c},\bm{c}'\in\{\bm{a}k,\bm{b}k\mid k\in \bm{i}_{\mathcal{B}_{\alpha}^j}\}$ then $\bm{c}$ and $\bm{c}'$ are in the same set of the partition
$G^{j+1}$ hence $\deg(p_{\bm{c}})=\deg(p_{\bm{c}'})$.
The fact that Equation (\ref{eq:star1}) holds for $\T_{j+2}$ means that in $\R[\Theta]_{\T_{j+2}}$
we have
\begin{align}
t(\bm{a}k_1)t(\bm{b}k_2) &= t_{\mathcal{B}^{j+1}_{1}}t_{\mathcal{B}^{j+1}_{2}}\label{eq:star2} \\ 
&= t(\bm{b}k_1)t(\bm{a}k_2) \nonumber
\end{align}
for some $\mathcal{B}^{j+1}_{1},\mathcal{B}^{j+1}_{2}\in \T_{G^{j+1}}$. Therefore $\deg(p_{\bm{a}k_1}p_{\bm{b}k_2})=
\deg(p_{\bm{b}k_1}p_{\bm{a}k_2})$ with respect to $\mathcal{A}_{j+1}$. This completes the
proof for $m=0$.

As a result, all the equations in $\Quad_{B_j}$ can be lifted to elements in $\ker(\varphi_{\T_{j+2}})$. Using the notation in Equation~(\ref{eq:star2}), the lift of an element $f = p_{\bm{a}k_1}p_{\bm{b}k_2}-p_{\bm{b}k_1}p_{\bm{a}k_2} \in \Quad_{B_j}$
depends on the assignment $\{\bm{a}k_1,\bm{a}k_2, \bm{b}k_1,\bm{b}k_2\}\to \{\mathcal{B}_1^{j+1},\mathcal{B}_2^{j+1}\}$.
For instance, if $t(\bm{a}k_1)=t(\bm{b}k_1)=t_{\mathcal{B}_{1}^{j+1}}$ and $t(\bm{b}k_2)=t(\bm{a}k_2)=
t_{\mathcal{B}_{2}^{j+1}}$ then $ T_{f}= \bm{i}_{\mathcal{B}_{1}^{j+1}}\times\bm{i}_{\mathcal{B}_{2}^{j+1}}$ so
\[\Lift(f)=\{f_{\beta}\mid \beta \in T_f\}=\{ p_{\bm{a}k_1\beta_1}^1p_{\bm{b}k_2\beta_2}^2-p_{\bm{b}k_1\beta_1}
^1p_{\bm{a}k_2\beta_2}^2 \mid \beta_1 \in \bm{i}_{\mathcal{B}_{1}^{j+1}}, \beta_2 \in \bm{i}_{\mathcal{B}_{2}^{j+1}} \}.\]


Suppose we have constructed $\Lift^{m-1}(\Quad_{B_j})$ inductively by lifting the equations in $\Quad_{B_j}$ and at each step all equations lift.
An element in $\Lift^{m-1}(\Quad_{B_j})$ is a binomial of the form
\begin{align}
f=p_{\bm{a}k_1s}p_{\bm{b}k_2u'}-p_{\bm{b}k_1u}p_{\bm{a}k_2s'} \label{eq:extended}
\end{align} where $\alpha \in \{1,\ldots,r_j\}, \bm{a},\bm{b}
\in G_{\alpha}^{j}$, $k_1,k_2\in \bm{i}_{\mathcal{B}^j_{\alpha}}$ and $s,s',u,u'$ are sequences of nonnegative integers
of length $m-1$ that arise as subindices after lifting $m-1$ times. Note that
$\aaa k_1s,\bbb k_2u',\bbb k_1 u, \aaa k_2s' \in \bm{i}_{\T_{j+m}}$The claim is that (\ref{eq:extended}) is $\mathcal{A}_{j+m}$-graded. 

Following a similar argument as for $m=0$, we know that two elements in
the same set of the 
partition $G^{j+m}$
have the same multidegree with respect to $\mathcal{A}_{j+m}$. As before, this condition can be verified
for $f$ by checking that
\begin{align} 
t(\bm{a}k_1s)t(\bm{b}k_2 u')=t(\bm{b}k_1u)t(\bm{a}k_2s') \mbox{ in } \R[\Theta]_{\T_{j+m+1}}.\label{eq:star3}
\end{align}
\mbox{For $\bm{c}\in \{\aaa k_1,\bbb k_2,
\aaa k_2,\bbb k_1\}$, }
$\left[\bm{c}\right]:=\{\beta\in \bm{i}
_{\T_{j+m}} \,|\,
\mbox{the root-to-$\beta$ path  in } \T_{j+m}  \mbox{ goes} \mbox{ trough } \bm{c}\}$.
To check that Equation~\ref{eq:star3} holds, consider (\ref{eq:star}) from  Definition~\ref{def:star} for the vertices $\aaa,\bbb \in G^j_{\alpha}$.  This equation is 
$t(\aaa k_1)t(\bbb k_2)=t(\bbb k_1)t(\aaa k_2)$ where $k_1,k_2\in \bm{i}_{\B_{\alpha}^j}$.
We use Lemma~\ref{lem:recursive} to rewrite this equation as
\begin{align}
 &\phantom{=}\left(\sum_{\bm{a}k_1s\in[\bm{a}k_1]}\left( \prod_{e\in E(\bm{a}k_1\to \bm{a}k_1s)}\!\!\!\!\!\!\!\!\theta(e)\right)t(\bm{a}k_1s)\right)
\cdot  \left(\sum_{\bm{b}k_2u'\in[\bm{b}k_2]}\left( \prod_{e\in E(\bm{b}k_2\to \bm{b}k_2u')}\!\!\!\!\!\!\!\!\theta(e)\right)t(\bm{b}k_2u')\right) \label{eq:star4}\\
&=\left(\sum_{\bm{b}k_1u\in[\bm{b}k_1]}\left( \prod_{e\in E(\bm{b}k_1\to \bm{b}k_1u)}\!\!\!\!\!\!\!\!\theta(e)\right)t(\bm{b}k_1u)\right)
\cdot  \left(\sum_{\bm{a}k_2s'\in[\bm{a}k_2]}\left( \prod_{e\in E(\bm{a}k_2\to \bm{a}k_2s')}\!\!\!\!\!\!\!\!\theta(e)\right)t(\bm{a}k_2s')\right).  \nonumber                        
\end{align}
When we specialize $t(\bm{a}k_1s)=t(\bm{b}k_2u')=t(\bm{b}k_1u)=t(\bm{a}k_2s')=1$ in each sum  in Equation~(\ref{eq:star4}) we
recover the interpolating polynomials for $t(\bm{a}k_1), t(\bm{b}k_2),t(\bm{b}k_1),t(\bm{a}k_2)$ in $\R[\Theta]_{\T_{j+m}}$.
By Lemma~\ref{lem:subbalanced}, $\T_{j+m}$ satisfies condition $(\star)$ therefore
\begin{align}
 &\phantom{=}\left(\sum_{\bm{a}k_1s\in[\bm{a}k_1]} \prod_{e\in E(\bm{a}k_1\to \bm{a}k_1s)}\!\!\!\!\!\!\!\!\theta(e)\right)
\cdot  \left(\sum_{\bm{b}k_2u'\in[\bm{b}k_2]} \prod_{e\in E(\bm{b}k_2\to \bm{b}k_2u')}\!\!\!\!\!\!\!\!\theta(e)\right) = \\
&\phantom{=} \left(\sum_{\bm{b}k_1u\in[\bm{b}k_1]} \prod_{e\in E(\bm{b}k_1\to \bm{b}k_1u)}\!\!\!\!\!\!\!\!\theta(e)\right)
\cdot  \left(\sum_{\bm{a}k_2s'\in[\bm{a}k_2]} \prod_{e\in E(\bm{a}k_2\to \bm{a}k_2s')}\!\!\!\!\!\!\!\!\theta(e)\right).                          
\end{align}
The factors in the above equality are sums of monomials all with coefficients equal to one. Thus for every pair
$\bm{a}k_1s\in[\bm{a}k_1],\bm{b}k_2u'\in[\bm{b}k_2]$ in the product of the left hand side of the equation, there exists
a pair $\bm{a}k_2s'\in[\bm{a}k_2],\bm{b}k_2u\in[\bm{b}k_1]$ in the product of the right hand side of the equation such that
\begin{align}
 &\phantom{=}\left( \prod_{e\in E(\bm{a}k_1\to \bm{a}k_1s)}\!\!\!\!\!\!\!\!\theta(e)\right)
\cdot  \left( \prod_{e\in E(\bm{b}k_2\to \bm{b}k_2u')}\!\!\!\!\!\!\!\!\theta(e)\right) =
\left(\prod_{e\in E(\bm{b}k_1\to \bm{b}k_1u)}\!\!\!\!\!\!\!\!\theta(e)\right)
\cdot  \left( \prod_{e\in E(\bm{a}k_2\to \bm{a}k_2s')}\!\!\!\!\!\!\!\!\theta(e)\right).                          
\end{align}
Hence condition $(\star)$ for the vertices $\aaa,\bbb$ in $\T_{j+m+1}$ can be rewritten as
\[\sum_{\bm{a}k_1s\in[\bm{a}k_1],\bm{b}k_2u'\in[\bm{b}k_2]}\left( \prod_{e\in E(\bm{a}k_1\to \bm{a}k_1s)}\!\!\!\!\!\!\!\!\theta(e)\right)
  \left( \prod_{e\in E(\bm{b}k_2\to \bm{b}k_2u')}\!\!\!\!\!\!\!\!\theta(e)\right)(t(\bm{a}k_1s)t(\bm{b}k_2 u')-t(\bm{b}k_1u)t(\bm{a}k_2s'))=0.\]
Since $\T_{j+m+1}$ is stratified, the variables involved in the factored monomials above are disjoint
from the variables involved in the factors of the form $t(\bm{a}k_1s)t(\bm{b}k_2 u')-t(\bm{b}k_1u)t(\bm{a}k_2s')$, therefore
the equation above holds if and only if $t(\bm{a}k_1s)t(\bm{b}k_2 u')-t(\bm{b}k_1u)t(\bm{a}k_2s')=0$ for each summand. This
proves that the elements in $\Lift^{m-1}(\Quad_{B_j})$ are
$\mathcal{A}_{j+m}$-graded.
\end{proof}

\begin{proof}[Proof of Theorem~\ref{thm:main}]
If $\T$ is stratified, then $\T$ is an iteratively constructed staged tree and
$\T=\T_n$ for some $n$.
Set $F_{n}=\Lift^{n-2}_{\mathcal{A}_n}(\Quad_{B_1})\cup \Lift^{n-3}_{\mathcal{A}_n}(\Quad_{B_1})
\cup \cdots \cup \Quad_{B_{n-1}}$.
We prove by induction on $n$ that $\ker(\varphi_{\T_n})$ is generated by $F_n$  and that
$F_n$ is a Gr\"obner basis  with squarefree intial ideal. The first
non-trivial case is $n=2$. We have $F_2=\Quad_{B_1}$ and from Proposition~10 in \cite{Sullivant2007}, 
$F_2$ is a Gr\"obner basis for the ideal $\ker(\varphi_{\T_2})=\ker(\varphi_{\T_1})\times_{
\mathcal{A}_1}\langle 0 \rangle$. 
Suppose the statement is true for $i$,
so the elements in $F_{i}$ are a Gr\"obner basis for $\ker(\varphi_{\T_i})$. Since $\T_n$ is balanced,
by Lemma~\ref{lem:subbalanced} the trees
$\T_i$ and $\T_{i+1}$ are also balanced.
Then from Proposition~\ref{prop:extensions} the elements in $F_i$ are $\mathcal{\A}_{i}$ homogeneous,
 so by \cite[Proposition 10]{Sullivant2007} the set $F_{i+1}$ is
a Gr\"obner basis for $\ker(\varphi_{\T_{i+1}})$. Since the elements in $F_n$ are all extensions
of elements in $\Quad_{B_j}$ for $j$ with $1\leq j \leq n-1$ we see that the leading terms of
these binomials are squarefree. Hence the initial ideal of $\langle F_n\rangle$ is squarefree.
\end{proof}

\begin{corollary}
Let $(\T,\theta)$ be a balanced and stratified staged tree. Fix $\Delta$ to be the polytope
defined by the convex hull of the lattice points in the exponent matrix of the map $\varphi_{\T}$.
Then $\Delta$ has a regular unimodular triangulation. In particular the toric variety defined
by $\ker(\varphi_{\T})$ is Cohen-Macaulay.
\end{corollary}
\begin{proof}
 The ideal $\ker(\varphi_{\T})$ has
a square free quadratic Gr\"obner basis with respect to a term order $\prec$. From 
\cite[Corollary 8.9]{Sturmfels96}, this induces a regular unimodular triangulation of $\Delta$.
\end{proof}
\section{Connections to discrete statistical models} \label{sec:stats}
Staged tree models are a class of graphical discrete statistical models introduced by Anderson and Smith in \cite{SmithAnderson}.  While
Bayesian networks and decomposable models are defined via conditional independence statements on random variables corresponding
to the vertices of a graph, staged tree models
encode independence relations on the events of an outcome space represented by a tree. In the statistical literature these models are also referred to as chain event graphs. We refer the reader to the book \cite{CEGbook} and to \cite{causalCEG} to find out more about their statistical properties, practical implementation and causal interpretation. In this section we give a formal definition of these models and recall 
 results from \cite{DG2019}
and \cite{CGorgenPhD}
about their defining equations.

Given a discrete random variable $X$ with state space $\{0,\ldots,n\}$, a probability distribution
on $X$ is a vector $(p_0,\ldots,p_n)\in \R^{n+1}$ where $p_i=P(X=i)$, $i\in \{0,\ldots,n\}$, $p_i\geq 0$ and $\sum_{i=0}^n p_i=1$. The open probability simplex \[\Delta_n^{\circ}=\{(p_0,\ldots,p_n)\in \R^{n+1}\,|\, p_i> 0 ,p_0+\ldots+p_n=1\}\]
consists of all the possible positive probability distributions for a discrete random variable with state space
$\{0,\ldots,n\}$. A \emph{discrete statistical model} is a subset of $\Delta_n^{\circ}$. In the next definition we associate a discrete statistical model to 
a given staged tree.

\begin{definition}\label{def:stagedtreemodel}
Let $(\T,\theta)$ be a staged tree and let $\bm{\theta}=(\theta(e) \mid\theta(e)\in \im (\theta)\,)$ be a vector of parameters where each entry is a label in $\la$. We define the parameter space
$\Theta_{\T}:= \{\,\bm{\theta} \,\mid\, \theta(e)\in (0,1)\,\mbox{ and  for all } \bm{a}\in V,\, \sum_{e\in E(\bm{a})}\theta(e)=1 \,
\}.$ 
Note that $\Theta_{\T}$ is a product of simplices.
A \emph{staged tree model} $\mathcal{M}_{(\T,\theta)}$ is the image of the map $\Psi_{\T}: \Theta_{\T}\to \Delta_{|\bm{i}_{\T}|-1}^{\circ}$defined by 
\[\bm{\theta}\mapsto p_{\bm{\theta}}=\left(\prod_{e\in E(\lambda_{\bm{j}})}\theta(e)\right)_{\bm{j}\in \bm{i}_{\T}}.\]
We can check that for every $\bm{\theta}\in
\Theta_{\T}$, $p_{\bm{\theta}}$ is a probability distribution  and therefore
$\Psi(\Theta_{\T})\subset \Delta_{|\bm{i}_{\T}|-1}^{\circ}$. Two staged trees $(\T,\theta)$ and 
$(\T',\theta')$ are
said to be \emph{statistically equivalent} if there exists a bijection between the sets $\Lambda_{\T}$
and $\Lambda_{\T'}$ in such a way that the image of $\Psi_{\T}$ is equal to the image of $\Psi_{\T'}$
under this bijection.
\end{definition}

\begin{example}\label{ex:dec1}
The staged tree $\T_1$ in Figure~\ref{fig:trees} is the staged tree representation of
the decomposable model associated to the undirected graph $G= [12][23][34]$ on four nodes.
\end{example}

\begin{remark}
For staged tree models, the  root-to-leaf paths in the tree represent the possible unfoldings of a sequence
of events. Given an edge $(v,w)$ in $\T$, the label $\theta(v,w)$ is the transition probability from 
$v$ to $w$ given arrival at $v$.
\end{remark}

\begin{remark}
A staged tree model $\M_{(\T,\theta)}$ is a discrete statistical model parameterized by polynomials. The domain of this model is a 
semialgebraic set given by a product of simplices. As a consequence the image of $\Psi_{\T}$ is also a semialgebraic
set. An important property of these models as noted in \cite{CGorgenPhD} is that the only inequality constraints of the
image of $\Psi_{\T}$ are the ones imposed by the probability simplex, namely $0\leq p_{\bm{j}}\leq 1$ for $\bm{j}\in 
\bm{i}_{\T}$ and $\sum_{\bm{j}\in \bm{i}_{\T}}p_{\bm{j}}=1$. 
\end{remark}

In Definition~\ref{def:toricideal} we defined the toric ideal associated to a staged tree
$(\T,\theta)$. Now we define the ideal associated to a staged tree model $\M_{(\T,\theta)}$.
For this we use the rings $\R[p]_{\T}$ and $\R[\Theta]_{\T}$ from Definition~\ref{def:toricideal}.
Consider the ideal $\frak{q}$ of $\R[\Theta]_{\T}$ generated by all sum-to-one conditions
$1-\sum_{e\in E(\bm{a})}\theta(e)$ for $\bm{a}\in V$ and let $\R[\Theta]_{\M_{\T}}:=\R[\Theta]_{\T}/
\frak{q}$. Denote by $\pi$ the canonical projection from $\R[\Theta]_{\T}$ to the
quotient ring $\R[\Theta]_{\M_{\T}}$.

\begin{definition}
Let $\M_{(\T,\theta)} $ be a staged tree model and set $\overline{\varphi}_{\T}:=\pi\circ \varphi_{\T}$. The ideal
$\ker(\overline{\varphi}_{\T})$ is the \emph{staged tree model ideal} associated to the  model $\mathcal{M}_{(\T,\theta)}$.
\end{definition}

From the definition it follows that for every staged tree $(\T,\theta)$, the toric staged tree ideal
is contained in the staged tree model ideal, i.e. $\ker(\varphi_{\T})\subset
\ker(\overline{\varphi}_{\T})$ . It is not true in general that these two ideals are equal \cite{DG2019}. However, Theorem 10 in \cite{DG2019} states
that if a staged tree $(\T,\theta)$ is balanced, then $\ker(\varphi_{\T})=
\ker(\overline{\varphi}_{\T})$. 
\begin{corollary}\label{cor:main}
If $(\T,\theta)$ is a balanced and stratified staged tree, then the ideal $\ker(\overline{\varphi}_{\T})$
has a quadratic Gr\"obner basis with
squarefree initial ideal.
\end{corollary}

\begin{example} \label{ex:dec2}
Consider the staged tree model defined by the tree in Figure~\ref{fig:trees} as in Example~\ref{ex:dec1}. 
Since this staged tree model is equal to the decomposable model given by $G=[12][23][34]$, from \cite{GMS}
we know it has a quadratic Gr\"obner basis. We recover the same result from the perspective of staged trees by using Corollary~\ref{cor:main}.
\end{example}
Corollary~\ref{cor:main} is relevant in statistics because of the connection of
Gr\"obner bases to sampling \cite{Aoki2012}. We presented Example~\ref{ex:dec2} where a  balanced and stratified staged tree represents an instance of a decomposable graphical model. We now provide more examples of staged tree models for
which Corollary~\ref{cor:main} holds. The first one is an explanation of the contraction 
axiom for conditional independence statements through the lens of staged trees.
Before we present our examples we do a quick overview of discrete conditional
independence models. 	Our exposition follows that in \cite[Chapter 4]{Sullivant2019}, for more details we refer the reader to 
\cite[Chapters 1,2,3]{Handbook} and 
\cite{LecturesAlgStat2008}.

Let $X=(X_1,\ldots,X_n)$ be a vector of discrete random variables, where $X_i$ has state spaces $[d_i]$ for $i\in [n]$. 
The vector $X$ has state space $\mathcal{X}=[d_1]\times\cdots\times[d_n]$ and we write $p_{u_1\cdots u_n}$
for the probability $P(X_1=u_1,\ldots,X_n=u_n)$. For each subset $A\subset [n]$, $X_A$  is the subvector
of $X$ indexed by the elements in $A$. Similarly, $ \mathcal{X}_{A}=\prod_{i\in A}[d_i]$ and for a vector $x\in \mathcal{X}$,
$x_A$ denotes the restriction of $x$ to the indexes in $A$.

\begin{definition}
Let
$A,B,C$ be pairwise disjoint subsets of $[n]$. The random vector $X_A$ is conditionally indpendent of $X_B$
given $X_C$ if for every $a\in X_A, b\in X_B$ and $c\in X_C$
\[P(X_A=a,X_B=b|X_C=c)=P(X_A=a|X_C=c)\cdot P(X_B=b|X_C=c)\]
The notation ${X_A \independent X_B \,|\, X_C}$ is used to denote that the random vector $X$ satisfies the
\emph{conditional independence statement} that $X_A$ is conditionally independent on $X_B$ given $X_C$. When $C$ is
the empty set this reduces to marginal independence between $X_A$ and $X_B$.
\end{definition}

If $\mathcal{C}$ is a list of conditional independence statements among variables in
a vector $X$, the \emph{conditional
independence model} $\mathcal{M}_{\mathcal{C}}$ is the set of all probability distributions
on $\mathcal{X}$ that satisfy the conditional independence statements in $\mathcal{C}$.
A conditional independence statement ${X_A \independent X_B \,|\, X_C}$ translates into
the condition that the joint probability distribution of the variables in $X$ 
satisfies a set of quadratic equations.
For elements $a\in X_A, b\in X_B$ and $c\in X_C$ we set $p_{a,b,c,+}=P(X_A=a,X_B=b,X_C=c)$.
\begin{proposition}[\cite{Sullivant2019}] \label{pro:conditional}
If $X$ is a discrete random vector then the independence statement ${X_A \independent X_B \,|\, X_C}$ holds for $X$ if and only if
\[p_{a_1,b_1,c,+}p_{a_2,b_2,c,+}-p_{a_1,b_2,c,+}p_{a_2,b_1,c,+}=0\]
for all $a_1,a_2\in \mathcal{X}_A$, $b_1,b_2\in \mathcal{X}_B$ and $c\in \mathcal{X}_C$.
\end{proposition}
The conditional independence ideal $I_{{A\independent B \,|\, C}}$ is the ideal generated
by all quadrics in Proposition~\ref{pro:conditional}. If $\mathcal{C}$ is
a list of conditional independence statements then we define $I_{\mathcal{C}}$
as the sum of all conditional independence ideals associated to statements in $\mathcal{C}$.

\begin{example}We consider the contraction axiom for positive distributions using staged tree
models.
Fix three discrete random variables $X_1,X_2,X_3$ with state spaces $[d_1+1],[d_2+1],[d_3+1]$
respectively. The contraction axiom states that the set of conditional independence statements
$\mathcal{C}=\{ {X_1 \indep X_2 \,|\,X_3}, {X_2\indep X_3}\}$ implies the  statement 
${X_2\indep (X_{1},X_3)}$. A primary decomposition of the ideal $I_{\mathcal{C}}$ was obtained
in \cite[Theorem 1]{agbn}. Here we provide a proof using staged trees, that one of the primary components of
$I_{\mathcal{C}}$
is the prime binomial ideal $I_{{X_2\indep (X_{1},X_3)}}$. As mentioned in \cite{agbn} this is a well known fact. First we explain how to represent the  two statements in $\mathcal{C}$ with
a staged tree. Consider the tree $\T$ in Figure~\ref{fig:moretrees}. This tree represents the state space of the vector
$(X_3,X_2,X_1)$ as a sequence of events where $X_3$ takes place first, $X_2$  second and
$X_1$  third. The vertices of $\T$ are indexed recursively as defined at the beginning of
Section~\ref{sec:tfps}. The statement ${X_2\indep X_3}$ is represented by the stage
consisting of the vertices $\{0,\ldots,d_3\}$, these are colored gray in $\T$.
The statement ${ X_1 \indep X_2 \,|\,X_3}$ is represented by the stages $S_{\!0},\ldots,S_{\!d_3}$
where $S_i=\{ij\,|\, j\in \{0,\ldots, d_2\}\}$ and $i\in\{0,\ldots,d_3\}$. These stages
	mean that for a given outcome of $X_3$ the unfolding of the event $X_2$ followed by $X_1$
	behaves as an independence model on two  random variables. In Figure~\ref{fig:moretrees}
	the stage $S_0$ is colored in pink and the stage $S_{\!d_3}$ is colored in purple.  Although the  
	gray vertices
	are not in the same position, we can easily check that $\T$ is balanced and
	stratified. Therefore
	$\ker(\varphi_{\T})$ has a quadratic Gr\"obner basis. Following the
	proof of  Theorem~\ref{thm:main} we can construct this basis explicitly. It consists
	of a set of quadratic equations given by the elements in $\Quad_{B_2}$ coming from
	the stages in $S_{\!0},\ldots,S_{\!d_3}$ and the lifts of the equations $\Quad_{B_1}$ coming from the
	stage $\{0,\ldots,d_3\}$. 
	If we swap the order of $X_1$ and $X_2$ in $\T$, we obtain the staged tree $\T'$ in
	Figure~\ref{fig:moretrees}. This tree represents the same statistical model as $\T$
	now with the unfolding of events $X_3,X_2,X_1$. The gray stages in $\T'$ represent
	the statement ${X_2\indep (X_{1},X_3)}$. Hence, after establishing the evident bijection between the
	leaves of $\T$ and $\T'$ we see that $I_{{X_2\indep (X_{1},X_3)}} = \ker(\varphi_{\T'})=\ker(\varphi_{\T})$.
\end{example}

\begin{figure}[t]
\begin{center}
\begin{tikzpicture}
\renewcommand{\xx}{1.6}
\renewcommand{\yy}{0.35}
\node at (0.5*\xx,13*\yy) {$\T\!\!:$};
\node at (1*\xx,14.5*\yy) {\scriptsize$X_3$};
\node at (2*\xx,14.5*\yy) {\scriptsize$X_2$};
\node at (3*\xx,14.5*\yy) {\scriptsize$X_1$};
\draw [dashed] 	(.75*\xx,-0.2*\yy) rectangle (1.25*\xx,14*\yy);
\draw [dashed] 	(1.75*\xx,-0.2*\yy) rectangle (2.25*\xx,14*\yy);
\draw [dashed] 	(2.75*\xx,-0.2*\yy) rectangle (3.25*\xx,14*\yy);

\node (r) at (1*\xx,7*\yy) {\bt};

\node (dd1) at (2*\xx,11.9*\yy) {\small $0$};
\node (dd2) at (2*\xx,3.9*\yy) {\small$d_3$};

\node (d1) at (2*\xx,11*\yy) {\stage{Gray}{$\phantom{;}$}};
\node (d2) at (2*\xx,3*\yy) {\stage{Gray}{$\phantom{;}$}};
\node (d3) at (2*\xx,5.5*\yy) {$\vdots$};
\node (d4) at (2*\xx,8.5*\yy) {$\vdots$};

\node (c1) at (3*\xx,13.4*\yy) {\small$00$};
\node (c2) at (3*\xx,10.4*\yy) {\small$0d_2$};
\node (c3) at (3*\xx,5.4*\yy) {\small$d_30$};
\node (c4) at (3*\xx,2.4*\yy) {\small$d_3d_2$};
\node (d3) at (3*\xx,11.5*\yy) {$\vdots$};
\node (d4) at (3*\xx,3.5*\yy) {$\vdots$};
\node (d3) at (3*\xx,7.5*\yy) {$\vdots$};

\node (c1) at (3*\xx,12.5*\yy) {\stage{Lavender}{$\phantom{;}$}};
\node (c2) at (3*\xx,9.5*\yy) {\stage{Lavender}{$\phantom{;}$}};
\node (c3) at (3*\xx,4.5*\yy) {\stage{Periwinkle}{$\phantom{;}$}};
\node (c4) at (3*\xx,1.5*\yy) {\stage{Periwinkle}{$\phantom{;}$}};

\node (b1) at (4*\xx,13.5*\yy) {\bt};
\node (b2) at (4*\xx,11*\yy) {\bt};
\node (b3) at (4*\xx,3*\yy) {\bt};
\node (b4) at (4*\xx,.5*\yy) {\bt};
\node (b5) at (4*\xx,12.5*\yy) {$\vdots$};
\node (b6) at (4*\xx,2*\yy) {$\vdots$};
\node (b7) at (4*\xx,7.5*\yy) {$\vdots$};

\node [right, xshift=5] at (b1) {$p_{000}$};
\node [right, xshift=5] at (b2) {$p_{0d_2d_1}$};
\node [right, xshift=5] at (b3) {$p_{d_3d_20}$};
\node [right, xshift=5] at (b4) {$p_{d_3d_2d_1}$};

\draw[->] (r) -- node [above] {} (d1);
\draw[->] (r) -- node [below] {} (d2);

\draw[->] (d1) -- node [above] {} (c1);
\draw[->] (d1) -- node [above] {} (c2);
\draw[->] (d2) -- node [above] {} (c3);
\draw[->] (d2) -- node [above] {} (c4);

\draw[->] (c1) -- node [above] {} (b1);
\draw[->] (c1) -- node [below] {} (b2);
\draw[->] (c4) -- node [above] {} (b3);
\draw[->] (c4) -- node [below] {} (b4);


\end{tikzpicture}
\quad
\begin{tikzpicture}
\renewcommand{\xx}{1.6}
\renewcommand{\yy}{0.35}
\node at (0.5*\xx,13*\yy) {$\T'\!\!:$};
\node at (1*\xx,14.5*\yy) {\scriptsize$X_3$};
\node at (2*\xx,14.5*\yy) {\scriptsize$X_1$};
\node at (3*\xx,14.5*\yy) {\scriptsize$X_2$};
\draw [dashed] 	(.75*\xx,-0.2*\yy) rectangle (1.25*\xx,14*\yy);
\draw [dashed] 	(1.75*\xx,-0.2*\yy) rectangle (2.25*\xx,14*\yy);
\draw [dashed] 	(2.75*\xx,-0.2*\yy) rectangle (3.25*\xx,14*\yy);

\node (r) at (1*\xx,7*\yy) {\bt};

\node (dd1) at (2*\xx,11.9*\yy) {\small $0$};
\node (dd2) at (2*\xx,3.9*\yy) {\small$d_3$};

\node (d1) at (2*\xx,11*\yy) {\stage{Lavender}{$\phantom{;}$}};
\node (d2) at (2*\xx,3*\yy) {\stage{Periwinkle}{$\phantom{;}$}};
\node (d3) at (2*\xx,5.5*\yy) {$\vdots$};
\node (d4) at (2*\xx,8.5*\yy) {$\vdots$};

\node (c1) at (3*\xx,13.4*\yy) {\small$00$};
\node (c2) at (3*\xx,10.4*\yy) {\small$0d_1$};
\node (c3) at (3*\xx,5.4*\yy) {\small$d_30$};
\node (c4) at (3*\xx,2.4*\yy) {\small$d_3d_1$};
\node (d3) at (3*\xx,11.5*\yy) {$\vdots$};
\node (d4) at (3*\xx,3.5*\yy) {$\vdots$};
\node (d3) at (3*\xx,7.5*\yy) {$\vdots$};

\node (c1) at (3*\xx,12.5*\yy) {\stage{Gray}{$\phantom{;}$}};
\node (c2) at (3*\xx,9.5*\yy) {\stage{Gray}{$\phantom{;}$}};
\node (c3) at (3*\xx,4.5*\yy) {\stage{Gray}{$\phantom{;}$}};
\node (c4) at (3*\xx,1.5*\yy) {\stage{Gray}{$\phantom{;}$}};

\node (b1) at (4*\xx,13.5*\yy) {\bt};
\node (b2) at (4*\xx,11*\yy) {\bt};
\node (b3) at (4*\xx,3*\yy) {\bt};
\node (b4) at (4*\xx,.5*\yy) {\bt};
\node (b5) at (4*\xx,12.5*\yy) {$\vdots$};
\node (b6) at (4*\xx,2*\yy) {$\vdots$};
\node (b7) at (4*\xx,7.5*\yy) {$\vdots$};


\draw[->] (r) -- node [above] {} (d1);
\draw[->] (r) -- node [below] {} (d2);

\draw[->] (d1) -- node [above] {} (c1);
\draw[->] (d1) -- node [above] {} (c2);
\draw[->] (d2) -- node [above] {} (c3);
\draw[->] (d2) -- node [above] {} (c4);

\draw[->] (c1) -- node [above] {} (b1);
\draw[->] (c1) -- node [below] {} (b2);
\draw[->] (c4) -- node [above] {} (b3);
\draw[->] (c4) -- node [below] {} (b4);


\end{tikzpicture}

\end{center} 
\caption{The staged trees $\T$ and $\T'$ are statistically equivalent, they represent the contraction 
axiom for three discrete random variables $X_1,X_2$ and $X_3$.}\label{fig:moretrees}
\end{figure}
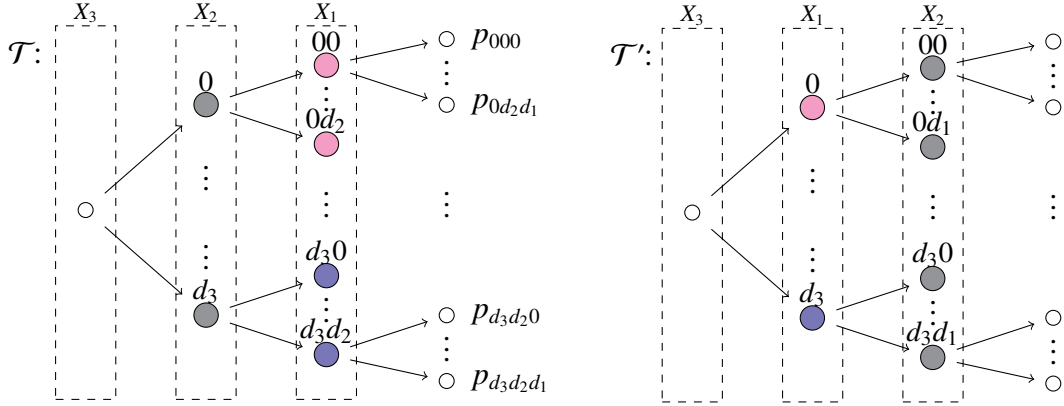

One of the main differences between staged tree models and discrete Bayesian networks is that the state
space of a Bayesian network is equal to the product of the state spaces of the random variables in the vertices
of the graph while the state space of a staged tree model does not necessarily have to equal a cartesian product.
When $\T$  is not equal to the cartesian product of some finite sets we call the tree $\T$ \emph{asymmetric}.
The lemmas that follow are important to show that Theorem~\ref{thm:main} also holds 
for the case when $\T$ is asymmetric. This implies that
we can use Theorem~\ref{thm:main} to construct quadratic Gr\"obner bases for  staged tree models whose 
underlying tree does not necessarily represents the distribution of
a vector of discrete random variables.

The definition of staged tree in \cite{CGorgenPhD} requires that each vertex in $\T$ has either no or at least
two outgoing edges from $v$. We stepped away from making this requirement for the staged trees 
we consider in Section~\ref{sec:stagedtrees}.
In the next lemmas we explain how this mild extension of the definition behaves with respect to
condition $(\star)$ and how trees defined according to \cite{CGorgenPhD} are recovered from the
more general trees we consider. Throughout the next lemmas, we fix a staged tree $(\T,\theta)$ with edge set $E$ and define
$E_1=\{e\in E\,\mid\, E(v)=\{e\} \text{ for some } v\in V\}$. For the trees in Figure~\ref{fig:asymtrees},
$\T$ has $|E_1|=6$ while for $\T'$, $|E_1|=0$.

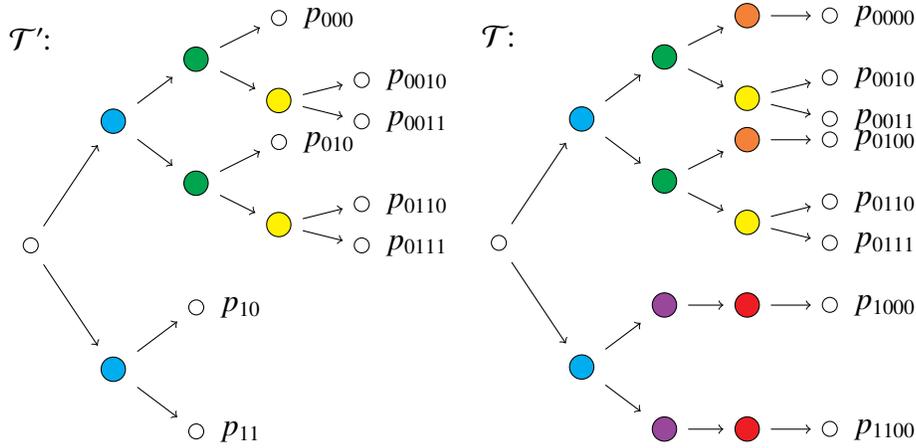
\begin{figure}[t]
\begin{center}
\quad
\begin{tikzpicture}
\renewcommand{\xx}{1.1}
\renewcommand{\yy}{0.55}

\node at (0,10*\yy) {$\T'\!\!:$};

\node (r) at (0,5*\yy) {\bt};

\node (v0) at (1*\xx,8*\yy) {\stage{Cyan}{$\phantom{;}$}};
\node (v1) at (1*\xx,2*\yy) {\stage{Cyan}{$\phantom{;}$}};

\node (w1) at (2*\xx,9.5*\yy) {\stage{Green}{$\phantom{;}$}};
\node (w2) at (2*\xx,6.5*\yy) {\stage{Green}{$\phantom{;}$}};
\node (w3) at (2*\xx,3.5*\yy) {\bt}; 
\node (w4) at (2*\xx,0.5*\yy) {\bt}; 

\node (a1) at (3*\xx,10.5*\yy) {\bt};
\node (a2) at (3*\xx,8.5*\yy) {\stage{Yellow}{$\phantom{;}$}};
\node (a3) at (3*\xx,7.5*\yy) {\bt}; 
\node (a4) at (3*\xx,5.5*\yy) {\stage{Yellow}{$\phantom{;}$}}; 

\node (b1) at (4*\xx,9*\yy) {\bt};
\node (b2) at (4*\xx,8*\yy) {\bt};
\node (b3) at (4*\xx,6*\yy) {\bt}; 
\node (b4) at (4*\xx,5*\yy) {\bt};

\draw[->] (r) -- node [above] {} (v0);
\draw[->] (r) -- node [above] {} (v1);
\draw[->] (v0) -- node [above] {} (w1);
\draw[->] (v0) -- node [above] {} (w2);
\draw[->] (v1) -- node [above] {} (w3);
\draw[->] (v1) -- node [above] {} (w4);

\draw[->] (w1) -- node [above] {} (a1);
\draw[->] (w1) -- node [above] {} (a2);

\draw[->] (w2) -- node [above] {} (a3);
\draw[->] (w2) -- node [above] {} (a4);

\draw[->] (a2) -- node [above] {} (b1);
\draw[->] (a2) -- node [above] {} (b2);

\draw[->] (a4) -- node [above] {} (b3);
\draw[->] (a4) -- node [above] {} (b4);

\node [right, xshift=5] at (a1) {$p_{000}$};
\node [right, xshift=5] at (a3) {$p_{010}$};
\node [right, xshift=5] at (b1) {$p_{0010}$};
\node [right, xshift=5] at (b2) {$p_{0011}$};
\node [right, xshift=5] at (b3) {$p_{0110}$};
\node [right, xshift=5] at (b4) {$p_{0111}$};
\node [right, xshift=5] at (w3) {$p_{10}$};
\node [right, xshift=5] at (w4) {$p_{11}$};

\end{tikzpicture}%
\,
\begin{tikzpicture}
\renewcommand{\xx}{1.1}
\renewcommand{\yy}{0.55}

\node at (0,10*\yy) {$\T\!\!:$};

\node (r) at (0,5*\yy) {\bt};

\node (v0) at (1*\xx,8*\yy) {\stage{Cyan}{$\phantom{;}$}};
\node (v1) at (1*\xx,2*\yy) {\stage{Cyan}{$\phantom{;}$}};

\node (w1) at (2*\xx,9.5*\yy) {\stage{Green}{$\phantom{;}$}};
\node (w2) at (2*\xx,6.5*\yy) {\stage{Green}{$\phantom{;}$}};
\node (w3) at (2*\xx,3.5*\yy) {\stage{Purple}{$\phantom{;}$}}; 
\node (w4) at (2*\xx,0.5*\yy) {\stage{Purple}{$\phantom{;}$}}; 

\node (a1) at (3*\xx,10.5*\yy) {\stage{Orange}{$\phantom{;}$}};
\node (a2) at (3*\xx,8.5*\yy) {\stage{Yellow}{$\phantom{;}$}};
\node (a3) at (3*\xx,7.5*\yy) {\stage{Orange}{$\phantom{;}$}}; 
\node (a4) at (3*\xx,5.5*\yy) {\stage{Yellow}{$\phantom{;}$}}; 
\node (a5) at (3*\xx,3.5*\yy) {\stage{Red}{$\phantom{;}$}}; 
\node (a6) at (3*\xx,0.5*\yy) {\stage{Red}{$\phantom{;}$}}; 

\node (b0) at (4*\xx,10.5*\yy) {\bt};
\node (b1) at (4*\xx,9*\yy) {\bt};
\node (b2) at (4*\xx,8*\yy) {\bt};
\node (b22) at (4*\xx,7.5*\yy) {\bt};
\node (b3) at (4*\xx,6*\yy) {\bt}; 
\node (b4) at (4*\xx,5*\yy) {\bt}; 
\node (b5) at (4*\xx,3.5*\yy) {\bt}; 
\node (b6) at (4*\xx,0.5*\yy) {\bt}; 

\node [right, xshift=5] at (b0) {$p_{0000}$};
\node [right, xshift=5] at (b1) {$p_{0010}$};
\node [right, xshift=5] at (b2) {$p_{0011}$};
\node [right, xshift=5] at (b22) {$p_{0100}$};
\node [right, xshift=5] at (b3) {$p_{0110}$};
\node [right, xshift=5] at (b4) {$p_{0111}$};
\node [right, xshift=5] at (b5) {$p_{1000}$};
\node [right, xshift=5] at (b6) {$p_{1100}$};

\draw[->] (r) -- node [above] {} (v0);
\draw[->] (r) -- node [above] {} (v1);
\draw[->] (v0) -- node [above] {} (w1);
\draw[->] (v0) -- node [above] {} (w2);
\draw[->] (v1) -- node [above] {} (w3);
\draw[->] (v1) -- node [above] {} (w4);

\draw[->] (w1) -- node [above] {} (a1);
\draw[->] (w1) -- node [above] {} (a2);

\draw[->] (w2) -- node [above] {} (a3);
\draw[->] (w2) -- node [above] {} (a4);
\draw[->] (w3) -- node [above] {} (a5);
\draw[->] (w4) -- node [above] {} (a6);

\draw[->] (a2) -- node [above] {} (b1);
\draw[->] (a2) -- node [above] {} (b2);

\draw[->] (a4) -- node [above] {} (b3);
\draw[->] (a4) -- node [above] {} (b4);

\draw[->] (a1) -- node [above] {} (b0);
\draw[->] (a3) -- node [above] {} (b22);
\draw[->] (a5) -- node [above] {} (b5);
\draw[->] (a6) -- node [above] {} (b6);

\end{tikzpicture}%

\end{center} 
\caption{The staged trees $\T$ and $\T'$ are statistically equivalent.}\label{fig:asymtrees}
\end{figure}

\begin{lemma}\label{lem:ttprime1}
Suppose $(\T,\theta)$ is a staged tree.
 Let $\T'$ be the staged tree obtained
from $\T$ by contracting the edges in $E_1$. Then $\mathcal{M}_{(\T,\theta)}=\mathcal{M}_{(\T',\theta)}$
and $\ker(\overline{\varphi}_\T)=\ker(\overline{\varphi}_{\T'})$.
\end{lemma}
\begin{proof}
First, note that the number of root-to-leaf paths in $\T'$ is the same as in $\T$. Moreover, each root-to-leaf
path $\lambda'$ in $\T'$ is obtained from a unique root-to-leaf path $\lambda$ in $\T$ by contracting the edges
in $E_1$. Now let $\lambda$ be a root-to-leaf path in $\T$. The $\lambda$-coordinate of the map $\Psi_{\T}$
applied to an element $\bm{\theta}\in \Theta_{\T}$ is
\begin{align*}
[\Psi_{\T}(\bm{\theta})]_{\lambda} &= \prod_{e\in E(\lambda)} \theta(e)= \prod_{e\in E(\lambda')}\theta(e) \\
&=[\Psi_{\T'}(\bm{\theta}\mid_{\T'})]_{\lambda'} 
\end{align*}
The second equality in the previous equation follows from taking a closer look at $\Theta_{\T}$. Indeed for all $e\in E_1$
we have $\theta(e)=1$ because of the sum-to-one conditions imposed on $\Theta_{\T}$ in Definition~\ref{def:stagedtreemodel}.
For the third equality, $\bm{\theta}\mid_{\T'}$ denotes the restriction of the vector $\bm{\theta}$ to the edge labels of
$\T'$. It follows from the equalities above that the coordinates of $\Psi_{\T}$ and $\Psi_{\T'}$ are equal. Therefore
$\mathcal{M}_{(\T,\theta)}=\mathcal{M}_{(\T',\theta)}$. A similar argument applied to the maps $\overline{\varphi}_{\T}$ and $\overline{\varphi}_{\T'}$
shows that $\ker(\overline{\varphi}_\T)=\ker(\overline{\varphi}_{\T'})$. To carry out this argument we need to reindex the leaves
of the trees, this can be done by dropping the index of the elements in $E_1$. 
\end{proof}
We illustrate Lemma~\ref{lem:ttprime1} in Figure~\ref{fig:asymtrees} where $\T'$ is obtained from
$\T$ by contracting the six edges in $E_1$. The two staged trees in this figure define the same statistical model.

\begin{remark}
To prove Corollary ~\ref{cor:main} we used \cite[Theorem 10]{DG2019}. The proof of Theorem 10 in \cite{DG2019} is presented for trees such that $E_1=0$. However the result still holds when $|E_1|>1$
because the ideal $\ipaths$ ( from \cite{DG2019}) is contained in $\ker(\varphi_{\T})$ in this case also, see \cite{DG2019} 
for more details. 
\end{remark}

\begin{lemma} \label{lem:ttprime2}
Suppose $(\T,\theta)$ is a balanced  and stratified staged tree. Let $\T'$ be the tree obtained from $\T$ by
contracting the edges in $E_1$. Then $(\T',\theta)$ is also balanced.
\end{lemma}
\begin{proof}
Suppose $\T$ is balanced and $\aaa,\bbb$ are in the same stage. Following the notation from Definition~\ref{def:star},  we have 
$t(\bm{a}i)t( \bm{b}j)=t(\bm{b}j)t(\bm{a}j) \text{ in } \R[\Theta]_{\T}, \text{ for all } i\neq j \in \{0,1,\ldots,k\}$. If we specialize $\theta(e)=1$ in this 
equation for all $e\in E_1$ and since $\T'$ is stratified, then $t(\bm{a}i)t( \bm{b}j)\mid_{\theta(e)=1,e\in E_1}=t(\bm{b}j)t(\bm{a}i)\mid_{\theta(e)=1,e\in E_1}$ 
in $\R[\Theta]_{\T'}$. Therefore $\T'$ is also balanced.
\end{proof}

\begin{corollary} \label{cor:main2}
Suppose $\T$ is a balanced and stratified staged tree with $|E_1|>1$. Let $\T'$ be the staged
tree obtained from $\T$ by contracting the edges in $E_1$. Then $\ker(\overline{\varphi}_{\T'})$
is a toric ideal with a quadratic Gr\"obner basis whose intial ideal is squarefree.
\end{corollary}
\begin{proof}
From Corollary~\ref{cor:main} it follows that $\ker(\overline{\varphi}_{\T})$ is a toric ideal with a
quadratic Gr\"obner basis and squarefree initial ideal. After an appropiate bijection, by Lemma~\ref{lem:ttprime1}, $\ker(\overline{\varphi}_{\T})=\ker(\overline{\varphi}_{\T'})$.
\end{proof}

We illustrate the result in Corollary~\ref{cor:main2} with an example.

\begin{example}
Fix $\T$ and $\T'$ to be the staged trees in Figure~\ref{fig:asymtrees}. The staged tree $\T'$
is considered in \cite[Section 6]{DG2019} as an example of the possible unfolding of events in a cell
culture. A thorough discussion of this example and its difference with other graphical models is
also contained in \cite[Section 6]{DG2019}.  Here we explain how to obtain a Gr\"obner basis for
$\ker(\varphi_{\T'})$ using Corollary~\ref{cor:main2}.
The tree $\T'$ is balanced and
statistically equivalent to $\T$. By Corollary~\ref{cor:main}, $\T$ has a quadratic Gr\"obner
basis with square free initial ideal. Using the lemmas preceeding this example, there is a bijection
between the root-to-leaf paths in $\T$ and $\T'$ thus $\R[p]_{\T}$ and
$\R[p]_{\T'}$ are isomorphic. Under
this isomorphism, the Gr\"obner basis for $\ker(\overline{\varphi}_{\T})$ is a Gr\"obner basis
for $\ker(\overline{\varphi}_{\T'})$ its generators are
\begin{align*}
 &p_{0111}p_{10}-p_{0011}p_{110},\,p_{0011}p_{0110}-p_{0010}p_{0111},\,p_{0110}p_{10}-p_{0010}p_{110},\, \\
 &p_{0010}p_{010}-p_{000}p_{0110},\,p_{0011}p_{010}-p_{000}p_{0111},\,p_{010}p_{10}-p_{000}p_{110}.
\end{align*}

\end{example}
\section*{Acknowledgements}
Computations using the symbolic algebra software {\em Macaulay2} \cite{M2} were crucial for the development of this paper. We thank Thomas Kahle for his support in the completion of this project and Bernd Sturmfels for a careful reading of an earlier version of this manuscript.
Both authors were supported by the Deutsche Forschungsgemeinschaft (314838170, GRK 2297 MathCoRe).

\bibliography{JAS_Sample}
\bibliographystyle{plain}

\medskip

\noindent
\footnotesize {\bf Authors' addresses:}


\noindent Lamprini Ananiadi, Otto-von-Guericke Universit\"at, Magdeburg,
\hfill {\tt lamprini.ananiadi@ovgu.de}

\noindent Eliana Duarte,  MPI-MiS ,Leipzig and Otto-von-Guericke Universit\"at, Magdeburg,
\hfill {\tt eliana.duarte@ovgu.de}

\end{document}